\documentclass[a4paper,11pt]{amsart}
\usepackage{graphicx,eepic,color}
\usepackage{amssymb,amstext,amsfonts,amscd}
\catcode`\@=11
\newcommand\lsection{\@startsection {section}{1}{\z@}%
                                   {-3.5ex \@plus -1ex \@minus -.2ex}%
                                   {1.0ex \@plus.2ex}%
                                   {\normalfont\large\bfseries}}
\catcode`\@=12
\makeatletter
\@addtoreset{equation}{section}

\makeatother
\newtheorem{thm}{Theorem}[section]
\newtheorem{cor}[thm]{Corollary}
\newtheorem{lem}[thm]{Lemma}
\newtheorem{prop}[thm]{Proposition}

\theoremstyle{definition}
\newtheorem{defn}[thm]{Definition}

\theoremstyle{definition}
\newtheorem{rem}[thm]{Remark}

\theoremstyle{definition}
\definecolor{orange}{rgb}{1,0.5,0}

\newtheorem{exam}[thm]{Example}
\newcommand{\BS}[1]{\boldsymbol{#1}}
\newcommand{\BB}[1]{\mathbb{#1}}

\newcommand{\pd}[2]{\frac{\partial#1}{\partial#2}}
\newcommand{\rank}{\mathop{\mathrm{rank}}\nolimits}

\newcommand{\Id}{\mathop{\mathrm{Id}}\nolimits}
\newcommand{\Cons}{\mathop{\mathrm{Cons}}\nolimits}

\newcommand{\sint}{\mathop{\textrm{\scriptsize $\mathop{\int}\nolimits$}}\nolimits}
\newcommand{\Lk}{{\mathrm{Lk}}}
\newcommand{\e}{{\textrm{even}}}
\newcommand{\odd}{{\textrm{odd}}}

\newcommand{\medsum}{\mathop{\textrm{\scriptsize $\mathop\sum$}}}

\allowdisplaybreaks[4]

\begin{document} 
\title{\LARGE On the topology of stable maps}
\author{Nicolas Dutertre and Toshizumi Fukui}
\address{Universit\'e de Provence, Centre de Math\'ematiques et Informatique,
39 rue Joliot-Curie,
13453 Marseille Cedex 13, France.}
\email{dutertre@cmi.univ-mrs.fr}
\address{Department of Mathematics, Faculty of Science, Saitama University, 255 Shimo-Okubo, Urawa 338, Japan.}
\email{tfukui@rimath.saitama-u.ac.jp}
\thanks{Mathematics Subject Classification (2000) : 57R20, 57R45, 57R70,
58C27  \\
N. Dutertre is supported by {\em Agence Nationale de la Recherche}
(reference ANR-08-JCJC-0118-01)}

\begin{abstract}
We investigate how Viro's integral calculus applies 
for the study of the topology of stable maps. 
We also discuss several applications to Morin maps 
and complex maps.
\end{abstract}
\maketitle
\markboth{Nicolas Dutertre and Toshizumi Fukui}
{On the topology of stable  maps }
\lsection{Introduction}\label{introduction}
It is well known that there is a deep relation between the topology of a
manifold and the topology of the critical locus of maps. The best
example of this fact is Morse Theory which gives the homotopy type of  a
compact manifold in terms of the Morse indices of the critical points of
a Morse function.  Let us mention other examples. 

R.~Thom \cite{Thom} proved that the Euler characteristic of a compact
manifold $M$ of dimension at least $2$ had the same parity as the number of
cusps of a generic map $f : M \to \mathbb{R}^2$. Latter H.~I.~Levine
\cite{Levine} improved this result giving an equality relating $\chi(M)$
and the critical set of $f$. In \cite{Fukuda1988}, T.~Fukuda generalized
Thom's result to Morin maps $f : M \to \mathbb{R}^p$ when $\dim M
\ge p$. He proved that:
\begin{equation}\label{ThomFukuda}
\chi(M) + \sum_{k=1}^p \chi (\overline{A_k(f)}) =0 \bmod 2,
\end{equation}
where $A_k(f)$ is the set of points $x$ in $M$ such that $f$ has a
singularity of type $A_k$ at $x$ 
(see Section \ref{Euler characteristics of local generic fibers} 
for the definition of $A_k$). 
Furthermore if $f$ has only fold points (i.e., singularities of
type $A_1$), then T.~Fukuda gave an equality relating $\chi(M)$ to the
critical set of $f$. T.~Fukuda's formulas were extended to the case of a
Morin mapping $f : M \to N$, where $\dim M \ge \dim N$, by O.~Saeki
\cite{Saeki}. 
When $\dim M=\dim N$, 
similar formulas were obtained by J.~R.~Quine \cite{Quine} and
I.~Nakai \cite{Nakai1996}. 
On the other hand, Y.~Yomdin \cite{Yomdin83} showed the equality among 
Euler characteristics of singular sets of holomorphic maps.  
As Y.~Yomdin and I.~Nakai showed in this context, 
the integral calculus due to O.~Viro \cite{Viro1990} is 
useful to find relations like \eqref{ThomFukuda} for stable maps.
In this paper, we investigate how Viro's integral calculus applies 
in enough wide setup. To do this we introduce the notion 
of local triviality at 
infinity and give some examples to illustrate this notion  
in section \ref{Local triviality at infinity}. 
T.~Ohmoto showed that Yomdin-Nakai's formula is generalized to
the statement in terms of characteristic class and discuss a relation
with Thom polynomial
in his lecture of the conference on the occasion of 70th birthday of
T.~Fukuda held on 20 July 2010.

We consider a stable map $f:M\to N$ 
between two smooth manifolds $M$ and $N$. We assume that $\dim
M \ge \dim N$, that $N$ is connected and that $M$ and $N$ have finite
topological types. We also assume that $f$ is locally trivial at
infinity (see Definition \ref{Def:localTriv}) 
and has finitely many singularity 
types. Then the singular set $\Sigma(f)$ of $f$ is decomposed into a
finite union $\sqcup_{\nu} \nu(f)$, where $\nu(f)$ is the set of
singular points of $f$ of type $\nu$. In Theorems \ref{RThm1},
\ref{RThm1B}, \ref{RThm1C} and \ref{RThm1D}, we establish several
formulas between the Euler characteristics with compact support of $M$,
$N$ and the $\nu(f)$'s. We apply them to maps having 
singularities of type $A_k$ or $D_k$  in Corollaries \ref{RMorin1}, \ref{RMorin11}, \ref{RMorin2},
\ref{RMorin3}, \ref{RMorin4} and \ref{RMorin5}.

In Section \ref{SecMorin} of this paper, 
we apply the results of Section \ref{SecSatble1} 
to Morin maps and we use the link between the Euler characteristic with
compact support  and the topological Euler characteristic to recover and
improve several 
results of T.~Fukuda, T.~Fukuda and G.~Ishikawa, I.~Nakai,
J.~Quine, O.~Saeki. 
We end the paper with some remarks in
the complex case in Section \ref{ComplexMap}.

%
\setcounter{tocdepth}{2}
\tableofcontents
\lsection{Viro's integral calculus}\label{Geometric integration}
In this section, we recall the method of integration with respect to a finitely-additive measure due to O.~Viro  \cite{Viro1988}.

Let $X$ be a set and $\mathcal S(X)$ denote a collection of subsets of $X$ 
which satisfies the following properties:
\begin{itemize}
\item If $A$, $B\in\mathcal S(X)$, then 
$A\cup B\in\mathcal S(X)$, $A\cap B\in\mathcal S(X)$.
\item If $A\in\mathcal S(X)$, then $X\setminus A\in\mathcal S(X)$.
\end{itemize}

Let $R$ be a commutative ring.  
Let $\mu_X:\mathcal S(X)\to R$ be a map which satisfies the following properties:
\begin{itemize}
\item If $A$ and $B$ are homeomorphic then 
$\mu_X(A)=\mu_X(B)$. 
\item
For $A, B\in\mathcal S(X)$, 
 $
\mu_X(A\cup B)=\mu_X(A)+\mu_X(B)-\mu_X(A\cap B).
$ 
\end{itemize}

\begin{exam}\label{EulerChar}
The Euler characteristic of the 
homology with compact support, denoted by $\chi_c$,  satisfies these conditions for $\mu_X$ 
with $R=\BB{Z}$.
The mod 2 Euler characteristic 
also satisfies these conditions for $\mu_X$ 
with $R=\BB{Z}/2\BB{Z}$.
\end{exam}

Let $\Cons(X,\mathcal S(X),R)$ (or $\Cons(X)$, for short) 
denote the set of finite $R$-linear 
combinations of characteristic functions $\BS{1}_A$ of elements $A$ of 
$\mathcal S(X)$. 
For $B\in \mathcal S(X)$ and $\varphi \in \Cons(X,\mathcal S(X),R)$, we define the integral of $\varphi$ over $B$ with respect to $\mu_X$, denoted by $\int_B \varphi d\mu_X$,
by:
$$
\sint_B\varphi(x)d\mu_X(x)=\medsum_A\lambda_A\,\mu_X(A\cap B)\quad 
\textrm{ where }\quad 
\varphi=\medsum_A\lambda_A\BS{1}_A.
$$
We remark that 
$\mu_X(B)=\int_Bd\mu_X$. 

Now we are going to state a Fubini type theorem for this integration. We
need to introduce some notations.

We say that $(\mathcal S(X),\mathcal S(Y))$ \textbf{fits to} the map $f:X\to Y$
if the following conditions hold: 
\begin{itemize}
\item If $A\in\mathcal S(X)$, then $f(A)\in\mathcal S(Y)$.
\item $f^{-1}(y)\in\mathcal S(X)$ for $y\in Y$. 
\item
For $A\in\mathcal S(X)$, $B\in\mathcal S(Y)$ with $f(A)=B$,   
if $f|_A:A\to B$ is a locally trivial fibration with fiber $F$, then: 
$$
\mu_X(A)=\mu_X(F)\mu_Y(B).
$$
\item For $A\in\mathcal S(X)$, 
there is 
a filtration $\emptyset =B_{-1}  \subset B_0\subset B_1\subset\cdots\subset B_l=Y$
with $B_i\in\mathcal S(Y)$ such that: 
$$
f|_{f^{-1}(B_i\setminus B_{i-1})\cap A}:
{f^{-1}(B_i\setminus B_{i-1})\cap A} \to B_i\setminus B_{i-1}\quad (i=0,1,\dots,l)
$$
is a locally trivial fibration.
\end{itemize}
\begin{lem}[Fubini's theorem]\label{Fubini}
For $\varphi \in\Cons(X)$ and $f : X \rightarrow Y$ such that $(\mathcal S(X), \mathcal S(Y))$ fits to $f$, we have:
$$
\sint_X\varphi(x)d\mu_X=\sint_Yf_*\varphi(y)d\mu_Y
\qquad\textrm{  where }
\quad f_*\varphi(y)=\sint_{f^{-1}(y)}\varphi(x)d\mu_X.
$$
\end{lem}
\begin{proof}
It is enough to show the case when $\varphi_X=\BS{1}_A$
for $A\in \mathcal S(X)$. So let us show that: 
$$
\mu_X(A)=\sint_Y\mu_X(A\cap f^{-1}(y))d\mu_Y. 
$$ 
We take 
a filtration $\emptyset \subset B_{-1} \subset  B_0\subset B_1\subset\cdots\subset B_l$
($B_i\in\mathcal S(Y)$) so that: 
$$
f|_{f^{-1}(B_i\setminus B_{i-1})\cap A}:
{f^{-1}(B_i\setminus B_{i-1})\cap A} \to B_i\setminus B_{i-1}\quad (i=0,2,\dots,l)
$$
is a locally trivial fibration with a fiber $F_i$.
Then we have: 
\begin{align*}
\mu_X(A)
=&\medsum_{i=0}^l\mu_X(f^{-1}(B_i\setminus B_{i-1})\cap A)
\qquad\textrm{(additivity of $\mu$)}\\
=&\medsum_{i=0}^l\mu_X(F_i)\mu_Y(B_i\setminus B_{i-1})
\qquad\textrm{(triviality of $f|_A$ on $B_i\setminus B_{i-1}$)}\\
=&\medsum_{i=0}^l\mu_X(F_i)\sint_{B_i\setminus B_{i-1}}d\mu_Y
\qquad\textrm{(definition of $\sint$)}\\
=&\medsum_{i=0}^l\sint_{B_i\setminus B_{i-1}}\mu_X(F_i)d\mu_Y\\
=&\medsum_{i=0}^l\sint_{B_i\setminus B_{i-1}}\mu_X(A\cap f^{-1}(y))d\mu_Y
\qquad(F_i=A\cap f^{-1}(y)\ \textrm{ for } y\in B_i\setminus B_{i-1})\\
=&\sint_Y\mu_X(A\cap f^{-1}(y))d\mu_Y
\qquad\textrm{(additivity of $\sint$)}.
\end{align*}
\end{proof}
\begin{cor}\label{Fubini2}
Set $X_i=\{x\in X \ \vert \ \varphi (x)=i\}$, and $Y_j=\{y\in Y \ \vert \ f_*\varphi(y)=j\}$.
Then we have:
$$
\medsum_{i} i\,\mu_X(X_i)=\medsum_jj\,\mu_Y(Y_j). 
$$
\end{cor}
\begin{proof}
This is clear, since: 
\begin{align*}
\sint_X\varphi(x)d\mu_X
=&\medsum_i\sint_{X_i}\varphi(x)d\mu_X
=\medsum_i\sint_{X_i}id\mu_X
=\medsum_i i\mu_X(X_i),\\
\sint_Yf_*\varphi(y)d\mu_Y
=&\medsum_j\sint_{Y_j}f_*\varphi(y)d\mu_Y
=\medsum_j\sint_{Y_j} j d\mu_Y
=\medsum_j j\mu_Y(Y_j).
\end{align*}
\end{proof}
\begin{cor}\label{Fubini3}
If $f_*\varphi$ is a constant $d$ on $y\in Y$, we have: 
$$
\medsum_{i} i\,\mu_X(X_i)=d\,\mu_Y(Y). 
$$
\end{cor}
In the sequel, we will apply O.~Viro's integral calculus to investigate 
topology of stable maps (see \cite{Nakai1996} 
and \cite{Nakai2000} for a similar strategy).

\lsection{Local triviality at infinity}\label{Local triviality at infinity}
In this section, we define the notion of local triviality at infinity
for a smooth map.

\begin{defn}\label{Def:localTriv}
Let $f:M\to N$ be a smooth map between two smooth manifolds. 
We say $f$ is 
\textbf{locally trivial at infinity at $y\in N$}
if there are a compact set $K$ in $M$ and 
an open neighborhood $D$ of $y$ 
such that
$
f:(M\setminus K)\cap f^{-1}(D)\to D
$
is a trivial fibration. 
We say $f$ is 
\textbf{locally trivial at infinity} if it is locally trivial at infinity at any $y\in N$. 
\end{defn}
Here are some examples of functions not locally trivial at infinity.
\begin{exam}[Broughton \cite{Broughton}]
Consider $f(x,y)=x(xy+1)$. The critical set $\Sigma(f)$ of $f$ is empty. 
For $t\ne0$, 
$$
f^{-1}(t)=\{y=(t-x)/x^2\}.
$$
We have $f^{-1}(t)=\BB{R}^*$, $f^{-1}(0)=\BB{R}\cup\BB{R}^*$ and 
$\chi_c(f^{-1}(t))=-2$, $\chi_c(f^{-1}(0))=-3$. 
So this example is not locally trivial at infinity on $t=0$. 
The level curves of $f$ with level $-1/2, 0, 1/2$ are shown in
 the figure.
The thick line shows the level $0$. 
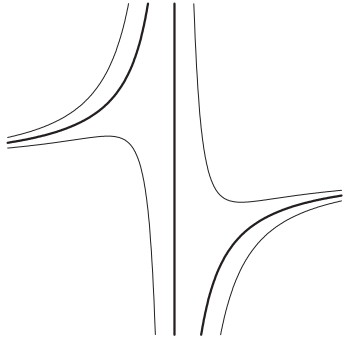
\begin{center}
\setlength{\unitlength}{2.5pt}
\begin{picture}(50,50)(-25,-25)
\path(-25,4.8)(-24.9,4.82)(-24.8,4.85)(-24.7,4.87)(-24.6,4.89)(-24.5,
4.91)(-24.4,4.94)(-24.3,4.96)(-24.2,4.99)(-24.1,5.01)(-24,5.03)(-23.9,
5.06)(-23.8,5.08)(-23.7,5.11)(-23.6,5.14)(-23.5,5.16)(-23.4,5.19)(-23.3,
5.21)(-23.2,5.24)(-23.1,5.27)(-23,5.29)(-22.9,5.32)(-22.8,5.35)(-22.7,
5.38)(-22.6,5.4)(-22.5,5.43)(-22.4,5.46)(-22.3,5.49)(-22.2,5.52)(-22.1,
5.55)(-22,5.58)(-21.9,5.61)(-21.8,5.64)(-21.7,5.67)(-21.6,5.7)(-21.5,
5.73)(-21.4,5.76)(-21.3,5.8)(-21.2,5.83)(-21.1,5.86)(-21,5.9)(-20.9,
5.93)(-20.8,5.96)(-20.7,6)(-20.6,6.03)(-20.5,6.07)(-20.4,6.1)(-20.3,
6.14)(-20.2,6.18)(-20.1,6.21)(-20,6.25)(-19.9,6.29)(-19.8,6.33)(-19.7,
6.36)(-19.6,6.4)(-19.5,6.44)(-19.4,6.48)(-19.3,6.52)(-19.2,6.56)(-19.1,
6.61)(-19,6.65)(-18.9,6.69)(-18.8,6.73)(-18.7,6.78)(-18.6,6.82)(-18.5,
6.87)(-18.4,6.91)(-18.3,6.96)(-18.2,7)(-18.1,7.05)(-18,7.1)(-17.9,
7.15)(-17.8,7.2)(-17.7,7.25)(-17.6,7.3)(-17.5,7.35)(-17.4,7.4)(-17.3,
7.45)(-17.2,7.5)(-17.1,7.56)(-17,7.61)(-16.9,7.67)(-16.8,7.72)(-16.7,
7.78)(-16.6,7.84)(-16.5,7.9)(-16.4,7.96)(-16.3,8.02)(-16.2,8.08)(-16.1,
8.14)(-16,8.2)(-15.9,8.27)(-15.8,8.33)(-15.7,8.4)(-15.6,8.46)(-15.5,
8.53)(-15.4,8.6)(-15.3,8.67)(-15.2,
8.74)(-15.1,8.82)(-15,8.89)(-14.9,
8.96)(-14.8,9.04)(-14.7,9.12)(-14.6,
9.19)(-14.5,9.27)(-14.4,9.36)(-14.3,
9.44)(-14.2,9.52)(-14.1,9.61)(-14,
9.69)(-13.9,9.78)(-13.8,9.87)(-13.7,
9.96)(-13.6,10.06)(-13.5,10.15)(-13.4,
10.25)(-13.3,10.35)(-13.2,10.45)(-13.1,
10.55)(-13,10.65)(-12.9,10.76)(-12.8,
10.86)(-12.7,10.97)(-12.6,11.09)(-12.5,
11.2)(-12.4,11.32)(-12.3,11.43)(-12.2,
11.56)(-12.1,11.68)(-12,11.81)(-11.9,
11.93)(-11.8,12.07)(-11.7,12.2)(-11.6,
12.34)(-11.5,12.48)(-11.4,12.62)(-11.3,
12.77)(-11.2,12.91)(-11.1,13.07)(-11,
13.22)(-10.9,13.38)(-10.8,13.55)(-10.7,
13.71)(-10.6,13.88)(-10.5,14.06)(-10.4,
14.24)(-10.3,14.42)(-10.2,14.61)(-10.1,
14.8)(-10,15)(-9.9,15.2)(-9.8,
15.41)(-9.7,15.62)(-9.6,15.84)(-9.5,
16.07)(-9.4,16.3)(-9.3,16.53)(-9.2,
16.78)(-9.1,17.03)(-9,17.28)(-8.9,
17.55)(-8.8,17.82)(-8.7,18.1)(-8.6,
18.39)(-8.5,18.69)(-8.4,18.99)(-8.3,
19.31)(-8.2,19.63)(-8.1,19.97)(-8,
20.31)(-7.9,20.67)(-7.8,21.04)(-7.7,
21.42)(-7.6,21.81)(-7.5,22.22)(-7.4,
22.64)(-7.3,23.08)(-7.2,23.53)(-7.1,
24)(-7,24.49)(-6.9,24.99)
\path(2.9,24.97)(3,22.22)(3.1,19.77)(3.2,17.58)(3.3,
15.61)(3.4,13.84)(3.5,12.24)(3.6,10.8)(3.7,
9.5)(3.8,8.31)(3.9,7.23)(4,6.25)(4.1,
5.35)(4.2,4.54)(4.3,3.79)(4.4,3.1)(4.5,
2.47)(4.6,1.89)(4.7,1.36)(4.8,0.87)(4.9,
0.42)(5,
0)(5.1,-0.38)(5.2,-0.74)(5.3,-1.07)
(5.4,-1.37)(5.5,-1.65)(5.6,-1.91)(5.7,
-2.15)(5.8,-2.38)(5.9,-2.59)(6,-2.78)
(6.1,-2.96)(6.2,-3.12)(6.3,-3.28)(6.4,
-3.42)(6.5,-3.55)(6.6,-3.67)(6.7,-3.79)
(6.8,-3.89)(6.9,-3.99)(7,-4.08)(7.1,
-4.17)(7.2,-4.24)(7.3,-4.32)(7.4,-4.38)
(7.5,-4.44)(7.6,-4.5)(7.7,-4.55)(7.8,-4.6)
(7.9,-4.65)(8,-4.69)(8.1,-4.72)(8.2,
-4.76)(8.3,-4.79)(8.4,-4.82)(8.5,-4.84)
(8.6,-4.87)(8.7,-4.89)(8.8,-4.91)(8.9,
-4.92)(9,-4.94)(9.1,-4.95)(9.2,-4.96)
(9.3,-4.97)(9.4,-4.98)(9.5,-4.99)(9.6,
-4.99)(9.7,-5)(9.8,-5)(9.9,-5)(10,
-5)(10.1,-5)(10.2,-5)(10.3,-5)(10.4,
-4.99)(10.5,-4.99)(10.6,-4.98)(10.7,-4.98)
(10.8,-4.97)(10.9,-4.97)(11,-4.96)(11.1,
-4.95)(11.2,-4.94)(11.3,-4.93)(11.4,-4.92)
(11.5,-4.91)(11.6,-4.9)(11.7,-4.89)(11.8,
-4.88)(11.9,-4.87)(12,-4.86)(12.1,-4.85)
(12.2,-4.84)(12.3,-4.83)(12.4,-4.81)(12.5,
-4.8)(12.6,-4.79)(12.7,-4.77)(12.8,-4.76)
(12.9,-4.75)(13,-4.73)(13.1,-4.72)(13.2,
-4.71)(13.3,-4.69)(13.4,-4.68)(13.5,-4.66)
(13.6,-4.65)(13.7,-4.64)(13.8,-4.62)(13.9,
-4.61)(14,-4.59)(14.1,-4.58)(14.2,-4.56)
(14.3,-4.55)(14.4,-4.53)(14.5,-4.52)(14.6,
-4.5)(14.7,-4.49)(14.8,-4.47)(14.9,-4.46)
(15,-4.44)(15.1,-4.43)(15.2,-4.41)(15.3,-4.4)
(15.4,-4.39)(15.5,-4.37)(15.6,-4.36)(15.7,
-4.34)(15.8,-4.33)(15.9,-4.31)(16,-4.3)
(16.1,-4.28)(16.2,-4.27)(16.3,-4.25)(16.4,
-4.24)(16.5,-4.22)(16.6,-4.21)(16.7,-4.2)
(16.8,-4.18)(16.9,-4.17)(17,-4.15)(17.1,
-4.14)(17.2,-4.12)(17.3,-4.11)(17.4,-4.1)
(17.5,-4.08)(17.6,-4.07)(17.7,-4.05)(17.8,
-4.04)(17.9,-4.03)(18,-4.01)(18.1,-4)
(18.2,-3.99)(18.3,-3.97)(18.4,-3.96)(18.5,
-3.94)(18.6,-3.93)(18.7,-3.92)(18.8,-3.9)
(18.9,-3.89)(19,-3.88)(19.1,-3.87)(19.2,
-3.85)(19.3,-3.84)(19.4,-3.83)(19.5,-3.81)
(19.6,-3.8)(19.7,-3.79)(19.8,-3.78)(19.9,
-3.76)(20,-3.75)(20.1,-3.74)(20.2,-3.73)
(20.3,-3.71)(20.4,-3.7)(20.5,-3.69)(20.6,
-3.68)(20.7,-3.66)(20.8,-3.65)(20.9,-3.64)
(21,-3.63)(21.1,-3.62)(21.2,-3.6)(21.3,-3.59)
(21.4,-3.58)(21.5,-3.57)(21.6,-3.56)(21.7,
-3.55)(21.8,-3.54)(21.9,-3.52)(22,-3.51)
(22.1,-3.5)(22.2,-3.49)(22.3,-3.48)(22.4,
-3.47)(22.5,-3.46)(22.6,-3.45)(22.7,-3.43)
(22.8,-3.42)(22.9,-3.41)(23,-3.4)(23.1,-3.39)
(23.2,-3.38)(23.3,-3.37)(23.4,-3.36)(23.5,
-3.35)(23.6,-3.34)(23.7,-3.33)(23.8,-3.32)
(23.9,-3.31)(24,-3.3)(24.1,-3.29)(24.2,-3.28)
(24.3,-3.27)(24.4,-3.26)(24.5,-3.25)(24.6,
-3.24)(24.7,-3.23)(24.8,-3.22)(24.9,-3.21)
(25,-3.2)
\path
(-25,3.2)(-24.9,3.21)(-24.8,
3.22)(-24.7,3.23)(-24.6,3.24)(-24.5,
3.25)(-24.4,3.26)(-24.3,3.27)(-24.2,
3.28)(-24.1,3.29)(-24,3.3)(-23.9,
3.31)(-23.8,3.32)(-23.7,3.33)(-23.6,
3.34)(-23.5,3.35)(-23.4,3.36)(-23.3,
3.37)(-23.2,3.38)(-23.1,3.39)(-23,
3.4)(-22.9,3.41)(-22.8,3.42)(-22.7,
3.43)(-22.6,3.45)(-22.5,3.46)(-22.4,
3.47)(-22.3,3.48)(-22.2,3.49)(-22.1,
3.5)(-22,3.51)(-21.9,3.52)(-21.8,
3.54)(-21.7,3.55)(-21.6,3.56)(-21.5,
3.57)(-21.4,3.58)(-21.3,3.59)(-21.2,
3.6)(-21.1,3.62)(-21,3.63)(-20.9,
3.64)(-20.8,3.65)(-20.7,3.66)(-20.6,
3.68)(-20.5,3.69)(-20.4,3.7)(-20.3,
3.71)(-20.2,3.73)(-20.1,3.74)(-20,
3.75)(-19.9,3.76)(-19.8,3.78)(-19.7,
3.79)(-19.6,3.8)(-19.5,3.81)(-19.4,
3.83)(-19.3,3.84)(-19.2,3.85)(-19.1,
3.87)(-19,3.88)(-18.9,3.89)(-18.8,
3.9)(-18.7,3.92)(-18.6,3.93)(-18.5,
3.94)(-18.4,3.96)(-18.3,3.97)(-18.2,
3.99)(-18.1,4)(-18,4.01)(-17.9,
4.03)(-17.8,4.04)(-17.7,4.05)(-17.6,
4.07)(-17.5,4.08)(-17.4,4.1)(-17.3,
4.11)(-17.2,4.12)(-17.1,4.14)(-17,
4.15)(-16.9,4.17)(-16.8,4.18)(-16.7,
4.2)(-16.6,4.21)(-16.5,4.22)(-16.4,
4.24)(-16.3,4.25)(-16.2,4.27)(-16.1,
4.28)(-16,4.3)(-15.9,4.31)(-15.8,
4.33)(-15.7,4.34)(-15.6,4.36)(-15.5,
4.37)(-15.4,4.39)(-15.3,4.4)(-15.2,
4.41)(-15.1,4.43)(-15,4.44)(-14.9,
4.46)(-14.8,4.47)(-14.7,4.49)(-14.6,
4.5)(-14.5,4.52)(-14.4,4.53)(-14.3,
4.55)(-14.2,4.56)(-14.1,4.58)(-14,
4.59)(-13.9,4.61)(-13.8,4.62)(-13.7,
4.64)(-13.6,4.65)(-13.5,4.66)(-13.4,
4.68)(-13.3,4.69)(-13.2,4.71)(-13.1,
4.72)(-13,4.73)(-12.9,4.75)(-12.8,
4.76)(-12.7,4.77)(-12.6,4.79)(-12.5,
4.8)(-12.4,4.81)(-12.3,4.83)(-12.2,
4.84)(-12.1,4.85)(-12,4.86)(-11.9,
4.87)(-11.8,4.88)(-11.7,4.89)(-11.6,
4.9)(-11.5,4.91)(-11.4,4.92)(-11.3,
4.93)(-11.2,4.94)(-11.1,4.95)(-11,
4.96)(-10.9,4.97)(-10.8,4.97)(-10.7,
4.98)(-10.6,4.98)(-10.5,4.99)(-10.4,
4.99)(-10.3,5)(-10.2,5)(-10.1,
5)(-10,5)(-9.9,5)(-9.8,5)(-9.7,
5)(-9.6,4.99)(-9.5,4.99)(-9.4,
4.98)(-9.3,4.97)(-9.2,4.96)(-9.1,
4.95)(-9,4.94)(-8.9,4.92)(-8.8,
4.91)(-8.7,4.89)(-8.6,4.87)(-8.5,
4.84)(-8.4,4.82)(-8.3,4.79)(-8.2,
4.76)(-8.1,4.72)(-8,4.69)(-7.9,
4.65)(-7.8,4.6)(-7.7,4.55)(-7.6,
4.5)(-7.5,4.44)(-7.4,4.38)(-7.3,
4.32)(-7.2,4.24)(-7.1,4.17)(-7,
4.08)(-6.9,3.99)(-6.8,3.89)(-6.7,
3.79)(-6.6,3.67)(-6.5,3.55)(-6.4,
3.42)(-6.3,3.28)(-6.2,3.12)(-6.1,
2.96)(-6,2.78)(-5.9,2.59)(-5.8,
2.38)(-5.7,2.15)(-5.6,1.91)(-5.5,
1.65)(-5.4,1.37)(-5.3,1.07)(-5.2,
0.74)(-5.1,0.38)(-5,
0)(-4.9,-0.42)(-4.8,-0.87)(-4.7,
-1.36)(-4.6,-1.89)(-4.5,-2.47)(-4.4,
-3.1)(-4.3,-3.79)(-4.2,-4.54)(-4.1,
-5.35)(-4,-6.25)(-3.9,-7.23)(-3.8,
-8.31)(-3.7,-9.5)(-3.6,-10.8)(-3.5,
-12.24)(-3.4,-13.84)(-3.3,-15.61)(-3.2,
-17.58)(-3.1,-19.77)(-3,-22.22)(-2.9,-24.97)
\path(6.9,-24.99)(7,-24.49)(7.1,-24)(7.2,
-23.53)(7.3,-23.08)(7.4,-22.64)(7.5,-22.22)
(7.6,-21.81)(7.7,-21.42)(7.8,-21.04)(7.9,
-20.67)(8,-20.31)(8.1,-19.97)(8.2,-19.63)
(8.3,-19.31)(8.4,-18.99)(8.5,-18.69)(8.6,
-18.39)(8.7,-18.1)(8.8,-17.82)(8.9,-17.55)
(9,-17.28)(9.1,-17.03)(9.2,-16.78)(9.3,
-16.53)(9.4,-16.3)(9.5,-16.07)(9.6,-15.84)
(9.7,-15.62)(9.8,-15.41)(9.9,-15.2)(10,
-15)(10.1,-14.8)(10.2,-14.61)(10.3,-14.42)
(10.4,-14.24)(10.5,-14.06)(10.6,-13.88)(10.7,
-13.71)(10.8,-13.55)(10.9,-13.38)(11,
-13.22)(11.1,-13.07)(11.2,-12.91)(11.3,
-12.77)(11.4,-12.62)(11.5,-12.48)(11.6,
-12.34)(11.7,-12.2)(11.8,-12.07)(11.9,
-11.93)(12,-11.81)(12.1,-11.68)(12.2,
-11.56)(12.3,-11.43)(12.4,-11.32)(12.5,
-11.2)(12.6,-11.09)(12.7,-10.97)(12.8,
-10.86)(12.9,-10.76)(13,-10.65)(13.1,
-10.55)(13.2,-10.45)(13.3,-10.35)(13.4,
-10.25)(13.5,-10.15)(13.6,-10.06)(13.7,
-9.96)(13.8,-9.87)(13.9,-9.78)(14,-9.69)
(14.1,-9.61)(14.2,-9.52)(14.3,-9.44)(14.4,
-9.36)(14.5,-9.27)(14.6,-9.19)(14.7,-9.12)
(14.8,-9.04)(14.9,-8.96)(15,-8.89)(15.1,
-8.82)(15.2,-8.74)(15.3,-8.67)(15.4,-8.6)
(15.5,-8.53)(15.6,-8.46)(15.7,-8.4)(15.8,
-8.33)(15.9,-8.27)(16,-8.2)(16.1,-8.14)
(16.2,-8.08)(16.3,-8.02)(16.4,-7.96)(16.5,
-7.9)(16.6,-7.84)(16.7,-7.78)(16.8,-7.72)
(16.9,-7.67)(17,-7.61)(17.1,-7.56)(17.2,-7.5)
(17.3,-7.45)(17.4,-7.4)(17.5,-7.35)(17.6,
-7.3)(17.7,-7.25)(17.8,-7.2)(17.9,-7.15)
(18,-7.1)(18.1,-7.05)(18.2,-7)(18.3,
-6.96)(18.4,-6.91)(18.5,-6.87)(18.6,-6.82)
(18.7,-6.78)(18.8,-6.73)(18.9,-6.69)(19,
-6.65)(19.1,-6.61)(19.2,-6.56)(19.3,-6.52)
(19.4,-6.48)(19.5,-6.44)(19.6,-6.4)(19.7,
-6.36)(19.8,-6.33)(19.9,-6.29)(20,-6.25)
(20.1,-6.21)(20.2,-6.18)(20.3,-6.14)(20.4,
-6.1)(20.5,-6.07)(20.6,-6.03)(20.7,-6)
(20.8,-5.96)(20.9,-5.93)(21,-5.9)(21.1,-5.86)
(21.2,-5.83)(21.3,-5.8)(21.4,-5.76)(21.5,
-5.73)(21.6,-5.7)(21.7,-5.67)(21.8,-5.64)
(21.9,-5.61)(22,-5.58)(22.1,-5.55)(22.2,
-5.52)(22.3,-5.49)(22.4,-5.46)(22.5,-5.43)
(22.6,-5.4)(22.7,-5.38)(22.8,-5.35)(22.9,
-5.32)(23,-5.29)(23.1,-5.27)(23.2,-5.24)
(23.3,-5.21)(23.4,-5.19)(23.5,-5.16)(23.6,
-5.14)(23.7,-5.11)(23.8,-5.08)(23.9,-5.06)
(24,-5.03)(24.1,-5.01)(24.2,-4.99)(24.3,
-4.96)(24.4,-4.94)(24.5,-4.91)(24.6,-4.89)
(24.7,-4.87)(24.8,-4.85)(24.9,-4.82)(25,-4.8)
\thicklines
\path(0,-25)(0,25)
\path(-25,4)(-24.9,4.02)(-24.8,4.03)(-24.7,4.05)(-24.6,4.07)
(-24.5,4.08)(-24.4,4.1)(-24.3,4.12)(-24.2,4.13)(-24.1,4.15)
(-24,4.17)(-23.9,4.18)(-23.8,4.2)(-23.7,4.22)(-23.6,4.24)(-23.5,4.26)
(-23.4,4.27)(-23.3,4.29)(-23.2,4.31)(-23.1,4.33)(-23,4.35)(-22.9,4.37)
(-22.8,4.39)(-22.7,4.41)(-22.6,4.42)(-22.5,4.44)(-22.4,4.46)(-22.3,4.48)
(-22.2,4.5)(-22.1,4.52)(-22,4.55)(-21.9,4.57)(-21.8,4.59)(-21.7,4.61)
(-21.6,4.63)(-21.5,4.65)(-21.4,4.67)(-21.3,4.69)(-21.2,4.72)(-21.1,4.74)
(-21,4.76)(-20.9,4.78)(-20.8,4.81)(-20.7,4.83)(-20.6,
4.85)(-20.5,4.88)(-20.4,4.9)(-20.3,
4.93)(-20.2,4.95)(-20.1,4.98)(-20,
5)(-19.9,5.03)(-19.8,5.05)(-19.7,
5.08)(-19.6,5.1)(-19.5,5.13)(-19.4,
5.15)(-19.3,5.18)(-19.2,5.21)(-19.1,
5.24)(-19,5.26)(-18.9,5.29)(-18.8,
5.32)(-18.7,5.35)(-18.6,5.38)(-18.5,
5.41)(-18.4,5.43)(-18.3,5.46)(-18.2,
5.49)(-18.1,5.52)(-18,5.56)(-17.9,
5.59)(-17.8,5.62)(-17.7,5.65)(-17.6,
5.68)(-17.5,5.71)(-17.4,5.75)(-17.3,
5.78)(-17.2,5.81)(-17.1,5.85)(-17,
5.88)(-16.9,5.92)(-16.8,5.95)(-16.7,
5.99)(-16.6,6.02)(-16.5,6.06)(-16.4,
6.1)(-16.3,6.13)(-16.2,6.17)(-16.1,
6.21)(-16,6.25)(-15.9,6.29)(-15.8,
6.33)(-15.7,6.37)(-15.6,6.41)(-15.5,
6.45)(-15.4,6.49)(-15.3,6.54)(-15.2,
6.58)(-15.1,6.62)(-15,6.67)(-14.9,
6.71)(-14.8,6.76)(-14.7,6.8)(-14.6,
6.85)(-14.5,6.9)(-14.4,6.94)(-14.3,
6.99)(-14.2,7.04)(-14.1,7.09)(-14,
7.14)(-13.9,7.19)(-13.8,7.25)(-13.7,
7.3)(-13.6,7.35)(-13.5,7.41)(-13.4,
7.46)(-13.3,7.52)(-13.2,7.58)(-13.1,
7.63)(-13,7.69)(-12.9,7.75)(-12.8,
7.81)(-12.7,7.87)(-12.6,7.94)(-12.5,
8)(-12.4,8.06)(-12.3,8.13)(-12.2,
8.2)(-12.1,8.26)(-12,8.33)(-11.9,
8.4)(-11.8,8.47)(-11.7,8.55)(-11.6,
8.62)(-11.5,8.7)(-11.4,8.77)(-11.3,
8.85)(-11.2,8.93)(-11.1,9.01)(-11,
9.09)(-10.9,9.17)(-10.8,9.26)(-10.7,
9.35)(-10.6,9.43)(-10.5,9.52)(-10.4,
9.62)(-10.3,9.71)(-10.2,9.8)(-10.1,
9.9)(-10,10)(-9.9,10.1)(-9.8,
10.2)(-9.7,10.31)(-9.6,10.42)(-9.5,
10.53)(-9.4,10.64)(-9.3,10.75)(-9.2,
10.87)(-9.1,10.99)(-9,11.11)(-8.9,
11.24)(-8.8,11.36)(-8.7,11.49)(-8.6,
11.63)(-8.5,11.76)(-8.4,11.9)(-8.3,
12.05)(-8.2,12.2)(-8.1,12.35)(-8,
12.5)(-7.9,12.66)(-7.8,12.82)(-7.7,
12.99)(-7.6,13.16)(-7.5,13.33)(-7.4,
13.51)(-7.3,13.7)(-7.2,13.89)(-7.1,
14.08)(-7,14.29)(-6.9,14.49)(-6.8,
14.71)(-6.7,14.93)(-6.6,15.15)(-6.5,
15.38)(-6.4,15.62)(-6.3,15.87)(-6.2,
16.13)(-6.1,16.39)(-6,16.67)(-5.9,
16.95)(-5.8,17.24)(-5.7,17.54)(-5.6,
17.86)(-5.5,18.18)(-5.4,18.52)(-5.3,
18.87)(-5.2,19.23)(-5.1,19.61)(-5,
20)(-4.9,20.41)(-4.8,20.83)(-4.7,
21.28)(-4.6,21.74)(-4.5,22.22)(-4.4,22.73)
(-4.3,23.26)(-4.2,23.81)(-4.1,24.39)(-4,25)
\path(4,-25)(4.1,-24.39)
(4.2,-23.81)(4.3,-23.26)(4.4,-22.73)(4.5,-22.22)
(4.6,-21.74)(4.7,-21.28)(4.8,-20.83)(4.9,-20.41)(5,-20)
(5.1,-19.61)(5.2,-19.23)(5.3,-18.87)(5.4,-18.52)
(5.5,-18.18)(5.6,-17.86)(5.7,-17.54)(5.8,-17.24)
(5.9,-16.95)(6,-16.67)(6.1,-16.39)(6.2,-16.13)
(6.3,-15.87)(6.4,-15.62)(6.5,-15.38)(6.6,-15.15)
(6.7,-14.93)(6.8,-14.71)(6.9,-14.49)(7,-14.29)
(7.1,-14.08)(7.2,-13.89)(7.3,-13.7)(7.4,-13.51)
(7.5,-13.33)(7.6,-13.16)(7.7,-12.99)(7.8,-12.82)(7.9,-12.66)
(8,-12.5)(8.1,-12.35)(8.2,-12.2)(8.3,-12.05)(8.4,-11.9)
(8.5,-11.76)(8.6,-11.63)(8.7,-11.49)(8.8,-11.36)(8.9,-11.24)
(9,-11.11)(9.1,-10.99)(9.2,-10.87)(9.3,-10.75)
(9.4,-10.64)(9.5,-10.53)(9.6,-10.42)(9.7,-10.31)
(9.8,-10.2)(9.9,-10.1)(10,-10)(10.1,-9.9)(10.2,-9.8)
(10.3,-9.71)(10.4,-9.62)(10.5,-9.52)(10.6,-9.43)(10.7,-9.35)
(10.8,-9.26)(10.9,-9.17)(11,-9.09)(11.1,-9.01)(11.2,-8.93)
(11.3,-8.85)(11.4,-8.77)(11.5,-8.7)(11.6,-8.62)(11.7,-8.55)
(11.8,-8.47)(11.9,-8.4)(12,-8.33)(12.1,-8.26)(12.2,-8.2)
(12.3,-8.13)(12.4,-8.06)(12.5,-8)(12.6,-7.94)(12.7,-7.87)
(12.8,-7.81)(12.9,-7.75)(13,-7.69)(13.1,-7.63)(13.2,-7.58)
(13.3,-7.52)(13.4,-7.46)(13.5,-7.41)(13.6,-7.35)(13.7,-7.3)
(13.8,-7.25)(13.9,-7.19)(14,-7.14)(14.1,-7.09)(14.2,-7.04)
(14.3,-6.99)(14.4,-6.94)(14.5,-6.9)(14.6,-6.85)(14.7,-6.8)
(14.8,-6.76)(14.9,-6.71)(15,-6.67)(15.1,-6.62)(15.2,-6.58)
(15.3,-6.54)(15.4,-6.49)(15.5,-6.45)(15.6,-6.41)(15.7,-6.37)
(15.8,-6.33)(15.9,-6.29)(16,-6.25)(16.1,-6.21)(16.2,-6.17)
(16.3,-6.13)(16.4,-6.1)(16.5,-6.06)(16.6,-6.02)(16.7,-5.99)
(16.8,-5.95)(16.9,-5.92)(17,-5.88)(17.1,-5.85)(17.2,-5.81)
(17.3,-5.78)(17.4,-5.75)(17.5,-5.71)(17.6,-5.68)(17.7,-5.65)
(17.8,-5.62)(17.9,-5.59)(18,-5.56)(18.1,-5.52)(18.2,-5.49)
(18.3,-5.46)(18.4,-5.43)(18.5,-5.41)(18.6,-5.38)(18.7,-5.35)
(18.8,-5.32)(18.9,-5.29)(19,-5.26)(19.1,-5.24)(19.2,-5.21)
(19.3,-5.18)(19.4,-5.15)(19.5,-5.13)(19.6,-5.1)(19.7,-5.08)
(19.8,-5.05)(19.9,-5.03)(20,-5)(20.1,-4.98)(20.2,-4.95)
(20.3,-4.93)(20.4,-4.9)(20.5,-4.88)(20.6,-4.85)(20.7,-4.83)
(20.8,-4.81)(20.9,-4.78)(21,-4.76)(21.1,-4.74)(21.2,-4.72)
(21.3,-4.69)(21.4,-4.67)(21.5,-4.65)(21.6,-4.63)(21.7,-4.61)
(21.8,-4.59)(21.9,-4.57)(22,-4.55)(22.1,-4.52)(22.2,-4.5)
(22.3,-4.48)(22.4,-4.46)(22.5,-4.44)(22.6,-4.42)(22.7,-4.41)
(22.8,-4.39)(22.9,-4.37)(23,-4.35)(23.1,-4.33)(23.2,-4.31)
(23.3,-4.29)(23.4,-4.27)(23.5,-4.26)(23.6,-4.24)(23.7,-4.22)
(23.8,-4.2)(23.9,-4.18)(24,-4.17)(24.1,-4.15)(24.2,-4.13)
(24.3,-4.12)(24.4,-4.1)(24.5,-4.08)(24.6,-4.07)(24.7,-4.05)
(24.8,-4.03)(24.9,-4.02)(25,-4)
\end{picture}
\end{center}
\end{exam}
A map $f:\BB{R}^2\to\BB{R}$ with $\Sigma(f)=\emptyset$ may not be
surjective. 
M.~Shiota remarked that 
the map $\BB{R}^2\to\BB{R}$, $(x,y)\mapsto(x(xy+1)+1)^2+x^2$, 
has empty critical set, and is not
surjective. 
\begin{exam}[Tib\u ar-Zaharia {\cite[Example 3.2]{TibarZaharia}}]
Consider $f(x,y)=x^2y^2+2xy+(y^2-1)^2$. 
Then $\Sigma(f)=\{(0,0),(1,-1),(-1,1)\}$ and 
$f(0,0)=1$, $f(1,-1)=f(-1,1)=-1$. 
Since $f^{-1}(t)$ is two lines (resp. circles) if $0\le t<1$ (resp. $-1<t<0$), 
we have: 
$$
\chi_c(f^{-1}(t))=
\begin{cases}
-2&(0\le t<1)\\
0&(-1<t<0).
\end{cases}
$$
So this example is not locally trivial at infinity on $t=0$. 
The level curves of $f$ with level $-1,-1/2, 0, 1/2, 1, 3/2$ 
are shown in the figure.
The thick line shows the level $0$. 
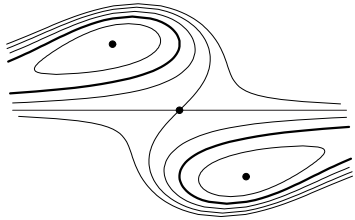
\begin{center}
\setlength{\unitlength}{2.5pt}
\begin{picture}(50,50)(-25,-25)
\path(-25,0)(25,0)
\put(0,0){\circle*{1}}
\put(10,-10){\circle*{1}}
\put(-10,10){\circle*{1}}
\path
(-18.48,5.41)(-15.22,5.57)(-13.7,5.72)(-12.51,5.87)
(-11.51,6.02)(-10.64,6.18)(-9.87,6.33)(-9.18,6.48)
(-8.55,6.64)(-7.99,6.79)(-7.47,6.94)(-7,7.1)(-6.56,7.25)
(-6.16,7.4)(-5.8,7.55)(-5.46,7.71)(-5.15,7.86)(-4.87,8.01)
(-4.61,8.17)(-4.37,8.32)(-4.15,8.47)(-3.95,8.63)(-3.77,8.78)
(-3.61,8.93)(-3.46,9.09)(-3.34,9.24)(-3.22,9.39)(-3.13,9.54)
(-3.05,9.7)(-2.98,9.85)(-2.93,10)(-2.89,10.16)(-2.87,10.31)
(-2.86,10.46)(-2.87,10.62)(-2.89,10.77)(-2.93,10.92)
(-2.98,11.08)(-3.05,11.23)(-3.14,11.38)(-3.25,11.53)
(-3.38,11.69)(-3.53,11.84)(-3.71,11.99)(-3.92,12.15)
(-4.17,12.3)(-4.47,12.45)(-4.83,12.61)(-5.29,12.76)
(-5.93,12.91)(-7.65,13.07)(-7.65,13.07)(-9.55,12.91)
(-10.38,12.76)(-11.03,12.61)(-11.59,12.45)(-12.09,12.3)
(-12.54,12.15)(-12.96,11.99)(-13.36,11.84)(-13.73,11.69)
(-14.09,11.53)(-14.43,11.38)(-14.76,11.23)(-15.08,11.08)
(-15.38,10.92)(-15.68,10.77)(-15.97,10.62)(-16.25,10.46)
(-16.53,10.31)(-16.8,10.16)(-17.06,10)(-17.32,9.85)
(-17.58,9.7)(-17.83,9.54)(-18.07,9.39)(-18.31,9.24)
(-18.55,9.09)(-18.78,8.93)(-19.01,8.78)(-19.23,8.63)
(-19.45,8.47)(-19.67,8.32)(-19.88,8.17)(-20.09,8.01)
(-20.29,7.86)(-20.49,7.71)(-20.67,7.55)(-20.86,7.4)
(-21.03,7.25)(-21.19,7.1)(-21.34,6.94)(-21.47,6.79)
(-21.58,6.64)(-21.67,6.48)(-21.73,6.33)(-21.74,6.18)
(-21.69,6.02)(-21.56,5.87)(-21.28,5.72)(-20.72,5.57)
(-18.48,5.41)
\path
(-25,1.1)(-22.89,1.19)(-17.59,1.49)(-13.95,1.79)(-11.25,
2.09)(-9.16,2.39)(-7.47,2.68)(-6.08,2.98)(-4.91,3.28)(-3.91,
3.58)(-3.04,3.88)(-2.28,4.18)(-1.62,4.47)(-1.03,4.77)(-0.52,
5.07)(-0.06,5.37)(0.34,5.67)(0.7,5.97)(1.01,6.26)(1.29,6.56)
(1.52,6.86)(1.73,7.16)(1.9,7.46)(2.04,7.76)(2.15,8.05)(2.24,
8.35)(2.3,8.65)(2.33,8.95)(2.34,9.25)(2.32,
9.55)(2.28,9.84)(2.21,10.14)(2.12,10.44)(2,10.74)(1.86,
11.04)(1.69,11.34)(1.48,11.63)(1.25,11.93)(0.98,12.23)(0.67,
12.53)(0.32,12.83)(-0.09,13.13)(-0.55,13.42)(-1.1,13.72)(-1.76,
14.02)(-2.58,14.32)(-3.72,14.62)(-6.7,14.92)
(-6.7,14.92)(-9.96,14.62)(-11.38,14.32)(-12.5,14.02)(-13.47,
13.72)(-14.34,13.42)(-15.15,13.13)(-15.91,12.83)(-16.63,12.53)
(-17.33,12.23)(-18.01,11.93)(-18.67,11.63)(-19.33,11.34)(-19.98,
11.04)(-20.63,10.74)(-21.28,10.44)(-21.93,10.14)(-22.6,9.84)
(-23.27,9.55)(-23.96,9.25)(-24.68,8.95)(-25.41,8.65)
\path
(0,0)(0.31,0.31)(0.62,0.62)(0.92,0.93)(1.22,1.24)(1.52,1.55)
(1.8,1.86)(2.08,2.18)(2.34,2.49)(2.59,2.8)(2.83,3.11)
(3.06,3.42)(3.27,3.73)(3.47,4.04)(3.65,4.35)
(3.82,4.66)(3.97,4.97)(4.1,5.28)(4.22,5.59)
(4.32,5.9)(4.41,6.22)(4.48,6.53)(4.54,6.84)
(4.57,7.15)(4.6,7.46)(4.6,7.77)(4.59,8.08)
(4.56,8.39)(4.52,8.7)(4.46,9.01)(4.38,9.32)
(4.28,9.63)(4.16,9.94)(4.03,10.25)(3.88,10.57)
(3.7,10.88)(3.5,11.19)(3.28,11.5)(3.03,11.81)
(2.76,12.12)(2.45,12.43)(2.11,12.74)(1.74,13.05)
(1.32,13.36)(0.84,13.67)(0.3,13.98)(-0.32,14.29)
(-1.05,14.61)(-1.97,14.92)(-3.21,15.23)(-6.44,15.54)
(-6.44,15.54)(-9.92,15.23)(-11.44,14.92)(-12.64,14.61)
(-13.67,14.29)(-14.61,13.98)(-15.47,13.67)(-16.28,13.36)
(-17.06,13.05)(-17.81,12.74)(-18.54,12.43)(-19.26,12.12)
(-19.97,11.81)(-20.67,11.5)(-21.38,11.19)(-22.09,10.88)
(-22.8,10.57)(-23.53,10.25)(-24.28,9.94)(-25.04,9.63)(-25.83,9.32)
\path
(24.1,0.96)(18.52,1.29)(15.27,1.61)(13.18,1.93)(11.75,2.25)
(10.73,2.57)(9.98,2.89)(9.41,3.21)(8.97,3.53)(8.62,3.86)
(8.34,4.18)(8.12,4.5)(7.93,4.82)(7.77,5.14)(7.63,5.46)(7.51,5.78)
(7.39,6.11)(7.28,6.43)(7.18,6.75)(7.07,7.07)(6.96,7.39)
(6.85,7.71)(6.73,8.03)(6.61,8.35)(6.47,8.68)(6.33,9)(6.18,9.32)
(6.01,9.64)(5.83,9.96)(5.64,10.28)(5.43,10.6)(5.21,10.92)
(4.97,11.25)(4.71,11.57)(4.43,11.89)(4.12,12.21)(3.79,12.53)
(3.43,12.85)(3.03,13.17)(2.6,13.5)(2.13,13.82)(1.6,14.14)
(1,14.46)(0.32,14.78)(-0.48,15.1)(-1.47,15.42)(-2.8,15.74)
(-6.22,16.07)(-6.22,16.07)(-9.9,15.74)(-11.5,15.42)(-12.76,15.1)
(-13.85,14.78)(-14.83,14.46)(-15.74,14.14)(-16.6,13.82)(-17.42,13.5)
(-18.21,13.17)(-18.99,12.85)(-19.75,12.53)(-20.5,12.21)
(-21.25,11.89)(-22,11.57)(-22.75,11.25)(-23.52,10.92)(-24.3,10.6)
(-25.09,10.28)(-25.91,9.96)
\path
(18.48,-5.41)(15.22,-5.57)(13.7,-5.72)(12.51,-5.87)
(11.51,-6.02)(10.64,-6.18)(9.87,-6.33)(9.18,-6.48)
(8.55,-6.64)(7.99,-6.79)(7.47,-6.94)(7,-7.1)(6.56,-7.25)
(6.16,-7.4)(5.8,-7.55)(5.46,-7.71)(5.15,-7.86)(4.87,-8.01)
(4.61,-8.17)(4.37,-8.32)(4.15,-8.47)(3.95,-8.63)(3.77,-8.78)
(3.61,-8.93)(3.46,-9.09)(3.34,-9.24)(3.22,-9.39)(3.13,-9.54)
(3.05,-9.7)(2.98,-9.85)(2.93,-10)(2.89,-10.16)(2.87,-10.31)
(2.86,-10.46)(2.87,-10.62)(2.89,-10.77)(2.93,-10.92)
(2.98,-11.08)(3.05,-11.23)(3.14,-11.38)(3.25,-11.53)
(3.38,-11.69)(3.53,-11.84)(3.71,-11.99)(3.92,-12.15)
(4.17,-12.3)(4.47,-12.45)(4.83,-12.61)(5.29,-12.76)
(5.93,-12.91)(7.65,-13.07)(7.65,-13.07)(9.55,-12.91)
(10.38,-12.76)(11.03,-12.61)(11.59,-12.45)(12.09,-12.3)
(12.54,-12.15)(12.96,-11.99)(13.36,-11.84)(13.73,-11.69)
(14.09,-11.53)(14.43,-11.38)(14.76,-11.23)(15.08,-11.08)
(15.38,-10.92)(15.68,-10.77)(15.97,-10.62)(16.25,-10.46)
(16.53,-10.31)(16.8,-10.16)(17.06,-10)(17.32,-9.85)
(17.58,-9.7)(17.83,-9.54)(18.07,-9.39)(18.31,-9.24)
(18.55,-9.09)(18.78,-8.93)(19.01,-8.78)(19.23,-8.63)
(19.45,-8.47)(19.67,-8.32)(19.88,-8.17)(20.09,-8.01)
(20.29,-7.86)(20.49,-7.71)(20.67,-7.55)(20.86,-7.4)
(21.03,-7.25)(21.19,-7.1)(21.34,-6.94)(21.47,-6.79)
(21.58,-6.64)(21.67,-6.48)(21.73,-6.33)(21.74,-6.18)
(21.69,-6.02)(21.56,-5.87)(21.28,-5.72)(20.72,-5.57)
(18.48,-5.41)
\path
(25,-1.1)(22.89,-1.19)(17.59,-1.49)(13.95,-1.79)(11.25,-
2.09)(9.16,-2.39)(7.47,-2.68)(6.08,-2.98)(4.91,-3.28)(3.91,-
3.58)(3.04,-3.88)(2.28,-4.18)(1.62,-4.47)(1.03,-4.77)(0.52,-
5.07)(0.06,-5.37)(-0.34,-5.67)(-0.7,-5.97)(-1.01,-6.26)(-1.29,-6.56)
(-1.52,-6.86)(-1.73,-7.16)(-1.9,-7.46)(-2.04,-7.76)(-2.15,-8.05)(-2.24,-
8.35)(-2.3,-8.65)(-2.33,-8.95)(-2.34,-9.25)(-2.32,-
9.55)(-2.28,-9.84)(-2.21,-10.14)(-2.12,-10.44)(-2,-10.74)(-1.86,-
11.04)(-1.69,-11.34)(-1.48,-11.63)(-1.25,-11.93)(-0.98,-12.23)(-0.67,-
12.53)(-0.32,-12.83)(0.09,-13.13)(0.55,-13.42)(1.1,-13.72)(1.76,-
14.02)(2.58,-14.32)(3.72,-14.62)(6.7,-14.92)
(6.7,-14.92)(9.96,-14.62)(11.38,-14.32)(12.5,-14.02)(13.47,-
13.72)(14.34,-13.42)(15.15,-13.13)(15.91,-12.83)(16.63,-12.53)
(17.33,-12.23)(18.01,-11.93)(18.67,-11.63)(19.33,-11.34)(19.98,-
11.04)(20.63,-10.74)(21.28,-10.44)(21.93,-10.14)(22.6,-9.84)
(23.27,-9.55)(23.96,-9.25)(24.68,-8.95)(25.41,-8.65)
\path
(0,-0)(-0.31,-0.31)(-0.62,-0.62)(-0.92,-0.93)(-1.22,-1.24)(-1.52,-1.55)
(-1.8,-1.86)(-2.08,-2.18)(-2.34,-2.49)(-2.59,-2.8)(-2.83,-3.11)
(-3.06,-3.42)(-3.27,-3.73)(-3.47,-4.04)(-3.65,-4.35)
(-3.82,-4.66)(-3.97,-4.97)(-4.1,-5.28)(-4.22,-5.59)
(-4.32,-5.9)(-4.41,-6.22)(-4.48,-6.53)(-4.54,-6.84)
(-4.57,-7.15)(-4.6,-7.46)(-4.6,-7.77)(-4.59,-8.08)
(-4.56,-8.39)(-4.52,-8.7)(-4.46,-9.01)(-4.38,-9.32)
(-4.28,-9.63)(-4.16,-9.94)(-4.03,-10.25)(-3.88,-10.57)
(-3.7,-10.88)(-3.5,-11.19)(-3.28,-11.5)(-3.03,-11.81)
(-2.76,-12.12)(-2.45,-12.43)(-2.11,-12.74)(-1.74,-13.05)
(-1.32,-13.36)(-0.84,-13.67)(-0.3,-13.98)(0.32,-14.29)
(1.05,-14.61)(1.97,-14.92)(3.21,-15.23)(6.44,-15.54)
(6.44,-15.54)(9.92,-15.23)(11.44,-14.92)(12.64,-14.61)
(13.67,-14.29)(14.61,-13.98)(15.47,-13.67)(16.28,-13.36)
(17.06,-13.05)(17.81,-12.74)(18.54,-12.43)(19.26,-12.12)
(19.97,-11.81)(20.67,-11.5)(21.38,-11.19)(22.09,-10.88)
(22.8,-10.57)(23.53,-10.25)(24.28,-9.94)(25.04,-9.63)(25.83,-9.32)
\path
(-24.1,-0.96)(-18.52,-1.29)(-15.27,-1.61)(-13.18,-1.93)(-11.75,-2.25)
(-10.73,-2.57)(-9.98,-2.89)(-9.41,-3.21)(-8.97,-3.53)(-8.62,-3.86)
(-8.34,-4.18)(-8.12,-4.5)(-7.93,-4.82)(-7.77,-5.14)(-7.63,-5.46)(-7.51,-5.78)
(-7.39,-6.11)(-7.28,-6.43)(-7.18,-6.75)(-7.07,-7.07)(-6.96,-7.39)
(-6.85,-7.71)(-6.73,-8.03)(-6.61,-8.35)(-6.47,-8.68)(-6.33,-9)(-6.18,-9.32)
(-6.01,-9.64)(-5.83,-9.96)(-5.64,-10.28)(-5.43,-10.6)(-5.21,-10.92)
(-4.97,-11.25)(-4.71,-11.57)(-4.43,-11.89)(-4.12,-12.21)(-3.79,-12.53)
(-3.43,-12.85)(-3.03,-13.17)(-2.6,-13.5)(-2.13,-13.82)(-1.6,-14.14)
(-1,-14.46)(-0.32,-14.78)(0.48,-15.1)(1.47,-15.42)(2.8,-15.74)
(6.22,-16.07)(6.22,-16.07)(9.9,-15.74)(11.5,-15.42)(12.76,-15.1)
(13.85,-14.78)(14.83,-14.46)(15.74,-14.14)(16.6,-13.82)(17.42,-13.5)
(18.21,-13.17)(18.99,-12.85)(19.75,-12.53)(20.5,-12.21)
(21.25,-11.89)(22,-11.57)(22.75,-11.25)(23.52,-10.92)(24.3,-10.6)
(25.09,-10.28)(25.91,-9.96)
{\thicklines
\path(-25.37,2.55)(-21.5,2.83)(-18.35,3.11)(-15.73,3.39)(-13.54,3.68)
(-11.68,3.96)(-10.08,4.24)(-8.7,4.53)(-7.5,4.81)(-6.45,5.09)
(-5.53,5.37)(-4.72,5.66)(-4,5.94)(-3.37,6.22)(-2.81,6.51)(-2.32,6.79)
(-1.89,7.07)(-1.52,7.35)(-1.19,7.64)(-0.91,7.92)(-0.67,8.2)(-0.47,8.49)
(-0.31,8.77)(-0.18,9.05)(-0.09,9.33)(-0.03,9.62)(0,9.9)
(-0.01,10.18)(-0.04,10.47)(-0.11,10.75)(-0.22,11.03)(-0.35,11.31)
(-0.53,11.6)(-0.74,11.88)(-1.01,12.16)
(-1.32,12.45)(-1.69,12.73)
(-2.14,13.01)(-2.7,13.29)(-3.41,13.58)(-4.4,13.86)(-7.07,14.14)
(-7.07,14.14)(-10.03,13.86)(-11.33,13.58)(-12.35,13.29)
(-13.23,13.01)(-14.02,12.73)(-14.75,12.45)(-15.44,12.16)(-16.09,11.88)
(-16.72,11.6)(-17.32,11.31)(-17.92,11.03)(-18.5,10.75)(-19.07,10.47)
(-19.64,10.18)(-20.2,9.9)(-20.77,9.62)(-21.34,9.33)
(-21.92,9.05)(-22.5,8.77)(-23.1,8.49)(-23.71,8.2)(-24.34,7.92)(-25,7.64)(-25.68,7.35)
\path(25.37,-2.55)(21.5,-2.83)(18.35,-3.11)(15.73,-3.39)(13.54,-3.68)
(11.68,-3.96)(10.08,-4.24)(8.7,-4.53)(7.5,-4.81)(6.45,-5.09)
(5.53,-5.37)(4.72,-5.66)(4,-5.94)(3.37,-6.22)(2.81,-6.51)(2.32,-6.79)
(1.89,-7.07)(1.52,-7.35)(1.19,-7.64)(0.91,-7.92)(0.67,-8.2)(0.47,-8.49)
(0.31,-8.77)(0.18,-9.05)(0.09,-9.33)(0.03,-9.62)(0,-9.9)
(0.01,-10.18)(0.04,-10.47)(0.11,-10.75)(0.22,-11.03)(0.35,-11.31)
(0.53,-11.6)(0.74,-11.88)(1.01,-12.16)
(1.32,-12.45)(1.69,-12.73)
(2.14,-13.01)(2.7,-13.29)(3.41,-13.58)(4.4,-13.86)(7.07,-14.14)
(7.07,-14.14)(10.03,-13.86)(11.33,-13.58)(12.35,-13.29)
(13.23,-13.01)(14.02,-12.73)(14.75,-12.45)(15.44,-12.16)(16.09,-11.88)
(16.72,-11.6)(17.32,-11.31)(17.92,-11.03)(18.5,-10.75)(19.07,-10.47)
(19.64,-10.18)(20.2,-9.9)(20.77,-9.62)(21.34,-9.33)
(21.92,-9.05)(22.5,-8.77)(23.1,-8.49)(23.71,-8.2)(24.34,-7.92)
(25,-7.64)(25.68,-7.35)}
\end{picture}
\end{center}
\end{exam}

\lsection{Euler characteristics of local generic fibers}
\label{Euler characteristics of local generic fibers}
In this section, we present a general method for the computations of the Euler characteristic of the Milnor fibers of a stable map-germ. We start with a lemma. 

  
\begin{lem}
Let $Y$ be a manifold and let $X$ be a set defined by: 
$$
X=\{(x,y)\in\BB{R}^p\times Y:
{x_1}^2+\dots+{x_p}^2=g(y)
\}
$$
where $g(y)$ is a smooth positive function. 
Then $X$ is a smooth manifold and:
$$
\chi_c(X)=\chi(S^{p-1})\chi_c(Y)
=(1-(-1)^p)\chi_c(Y).
$$
\end{lem}
\begin{proof} 
It is easy to check that $X$ is a manifold. To obtain the equality,
consider the map: 
$$
X\to Y,\  (x,y)\mapsto y. 
$$
This is a locally trivial fibration whose fiber is $S^{p-1}$. 
\end{proof}
\begin{exam}\label{EX1}
Let $X$ be the set defined by:
$$
X=\{(x,y)\in\BB{R}^p\times\BB{R}^q:
{x_1}^2+\cdots+{x_p}^2={y_1}^2+\cdots+{y_q}^2+1\}.
$$
Since $X\to\BB{R}^q$, $(x,y)\mapsto y$, is a locally trivial fibration whose 
 fiber is $S^{p-1}$, 
we have: 
$$
\chi_c(X)
=\chi_c(S^{p-1})\chi_c(\BB{R}^q)
=(1-(-1)^p)(-1)^q=(-1)^q-(-1)^{p+q}.
$$
\end{exam}
\begin{exam}\label{EX2}
Let $X$ be the set defined by: 
$$
X=\{(x,y)\in\BB{R}^p\times\BB{R}^q:
{x_1}^2+\cdots+{x_p}^2={y_1}^2+\cdots+{y_q}^2\}.
$$
Since $X\setminus\{0\}\to\BB{R}^q\setminus\{0\}$, $(x,y)\mapsto y$, is a locally trivial fibration whose 
 fiber is $S^{p-1}$, we have: 
\begin{align*}
\chi_c(X)=&\chi_c(\{0\})+\chi_c(S^{p-1})\chi_c(\BB{R}^q\setminus\{0\})\\
=&1+(1-(-1)^p)((-1)^q-1)\\
=&(-1)^p+(-1)^q-(-1)^{p+q}.
\end{align*}
\end{exam}

Next we will apply this lemma and these examples to the computation of Euler characteristics of local nearby fibers of stable map-germs. The general setting is the following.
Let $\widetilde{B}$  be a small open ball in $\BB{R}^n$ centered at $0$ 
and let $B'$ be a small open ball in $\BB{R}^{a+b}$ centered at $0$. 
We consider a map $f$ defined by:
\begin{equation}\label{PreNormalForm}
f:\widetilde{B}\times B'\times\BB{R}^h\to\BB{R}\times\BB{R}^j\times\BB{R}^h,\qquad
(x,z,c)\mapsto(g(x;c)+Q(z),g'(x;c),c)
\end{equation}
where 
$Q(z)={z_1}^2+\cdots+{z_a}^2-{z_{a+1}}^2-\cdots-{z_{a+b}}^2$. 
Remember that stable-germs are versal unfoldings, deleting constant terms,  
of a map-germ $x\mapsto(g(x;0),g'(x;0))$, called the genotype, 
and can be written in this form. 
(See \cite[Part I, 9.]{Arnold-GuseinZade-Varchenko})

We want to compute the Euler characteristic of a local generic fiber around the point $(0,0,0)$, namely the fiber $f^{-1}(\varepsilon,\varepsilon',c)$ for small $\varepsilon$ and $\varepsilon'$.
First we remark that $f^{-1}(\varepsilon,\varepsilon',c)$ is diffeomorphic to: 
$$
F=\{(x,z)\in B\times B':g(x;c)+Q(z)=\varepsilon\},
$$
where $B$ is the nonsingular subset of $\widetilde{B}$ 
defined by $g'(x;c)=\varepsilon'$.  
Note that dim $F=n-j+a+b-1$ and dim $B=n-j$.
\begin{lem}
We have: 
$$
\chi_c(F)=
\begin{cases}
\chi_c(B_0)&
\textrm{$a$ even, $b$ even }\\
\chi_c(B)+\chi_c(B_+)-\chi_c(B_-)
&\textrm{$a$  even, $b$ odd \ }\\
\chi_c(B)-\chi_c(B_+)+\chi_c(B_-)
&\textrm{$a$ odd,  $b$  even }\\
-2\chi_c(B)-\chi_c(B_0)
&\textrm{$a$  odd, $b$ odd \ }
\end{cases}
$$
where: 
\begin{align*}
B_+=&\{x\in B \ \vert \ g(x;c)>\varepsilon\},\\
B_0=&\{x\in B \ \vert \ g(x;c)=\varepsilon\},\\
B_-=&\{x\in B \ \vert \ g(x;c)<\varepsilon\}. 
\end{align*}
\end{lem}
Remark that $B$, $B_+$, $B_-$ and $B_0$ depend on 
$\varepsilon$, $\varepsilon'$, $c$ and it would be better to denote them by 
$B(\varepsilon,\varepsilon',c)$,
$B_+(\varepsilon,\varepsilon',c)$,
$B_-(\varepsilon,\varepsilon',c)$ and $B_0(\varepsilon,\varepsilon',c)$ respectively. 
But we keep the notation in the lemma for shortness.   
\begin{proof}
Consider the map:  
$
\varphi:F\to B, \ (y,z)\mapsto y
$.
The singular set of $\varphi$ is described by: 
$$
\rank
\begin{pmatrix}
g_{y_i}&Q_{z_k}\\
1&0
\end{pmatrix}<m+1
\qquad(\textrm{i.e., }\ Q_{z_1}=\cdots=Q_{z_{a+b}}=0),
$$
where $m=n-j$ and $(y_1,\dots,y_m)$ denotes a local coordinates system for $B$. 
Note that, with the standard notation, $\Sigma(\varphi)=\Sigma^{a+b}(\varphi)$. 
Now we consider the singular set of $\varphi|_{\Sigma(\varphi)}$, 
which is defined by:
$$
\rank
\begin{pmatrix}
g_{y_i}&Q_{z_k}\\
1&0\\
0&Q_{z_iz_j}
\end{pmatrix}<m+a+b. 
$$
Since $Q$ is quadratic, we have $\Sigma^{a+b,1}(\varphi)= \emptyset $ which means that $\varphi$ is a fold map. Moreover $\Sigma(\varphi)$ is included in $\varphi^{-1}(B_0)$. Hence 
$\varphi_{\vert \varphi^{-1}(B^+)}$ and $\varphi_{\vert \varphi^{-1}(B^-)}$ are locally trivial. Furthermore the decomposition $\varphi^{-1}(B_0) =\Sigma(\varphi) \sqcup (\varphi^{-1}(B_0) \setminus   \Sigma(\varphi))$ gives a Whitney stratification of $\varphi^{-1}(B_0)$ and $\varphi_{\vert \Sigma(\varphi)}$ and $\varphi_{\vert \varphi^{-1}(B_0) \setminus   \Sigma(\varphi)}$ have no critical point so $\varphi_{\vert \varphi^{-1}(B_0)}$ is also trivial by the Thom-Mather lemma. 

Using Examples \ref{EX1} and \ref{EX2}, we remark the following:
$$
\chi_c(\varphi^{-1}(x))=
\begin{cases}
(-1)^b-(-1)^{a+b}& x\in B_-\\
(-1)^a-(-1)^{a+b}& x\in B_+\\
(-1)^a+(-1)^b-(-1)^{a+b}
&x\in B_0
\end{cases}
$$
In other words, $\chi_c(\varphi^{-1}(x))$ is given by 
the following table: 
\begin{center}
\begin{tabular}{|c|c|c|c|}
\hline
&$x\in B_+$&$x\in B_-$&$x\in B_0$\\
\hline
$a$ even, $b$ even &0&0&1\\
$a$ even, $b$ odd \ &$2$&0&$1$\\
$a$ odd, \ $b$ even &0&$2$&$1$\\
$a$ odd, \ $b$ odd \ &$-2$&$-2$&$-3$\\
\hline
\end{tabular}
\end{center}
Therefore, using the local trivialities mentioned above, 
we conclude as follows: 
$$
\chi_c(F)=
\begin{cases}
\chi_c(B_0)&
\textrm{$a$  even, $b$ even }\\
2\chi_c(B_+)+\chi_c(B_0)=\chi_c(B)+\chi_c(B_+)-\chi_c(B_-)
&\textrm{$a$ even, $b$  odd \ }\\
2\chi_c(B_-)+\chi_c(B_0)=\chi_c(B)-\chi_c(B_+)+\chi_c(B_-)
&\textrm{$a$ odd, \ $b$  even }\\
-2\chi_c(B_+)-2\chi_c(B_-)-3\chi_c(B_0)=-2\chi_c(B)-\chi_c(B_0)%
&\textrm{$a$ odd, \ $b$  odd \ }
\end{cases}
$$
Here we use the fact 
$\chi_c(B_+)+\chi_c(B_-)+\chi_c(B_0)=\chi_c(B)$. 
\end{proof}
If $B$ is an open $m$-ball, then $\chi_c(B)=(-1)^m$ and we conclude that: 
$$
1+(-1)^{m+a+b}\chi_c(F)=
\begin{cases}
(-1)^m((-1)^m+\chi_c(B_0))&
\textrm{$a$ even, $b$  even }\\
-(-1)^m(\chi_c(B_+)-\chi_c(B_-))
&\textrm{$a$ even, $b$  odd \ }\\
(-1)^m(\chi_c(B_+)-\chi_c(B_-))
&\textrm{$a$   odd, \ $b$  even }\\
-(-1)^m((-1)^m+\chi_c(B_0))
&\textrm{$a$ odd, \ $b$ odd \ }
\end{cases}
$$
If $f$ is an unfolding of a function-germ (i.e., $m=n$, $j=0$), 
then $B$ is an open $m$-ball. 
\begin{defn}\label{DefSsigma}
We consider an unfolding of a function-germ $(x,z)\mapsto g(x;0)+Q(z)$.
Let $\sigma$ denote the singularity type of the map 
$x\mapsto g(x,0)$. 
When $m+a+b$ is even, define $s_\sigma$ by: 
$$
s_\sigma=1+\chi_c(F)=
\begin{cases}
-\chi_c(B_+)+\chi_c(B_-)&\textrm{if $m$ is odd and $a+b$ is odd}\\
1+\chi_c(B_0)&\textrm{if $m$ is even and $a+b$ is even}
\end{cases}
$$
When $m+a+b$ is odd, define $s^{\max}_\sigma$, $s^{\min}_\sigma$ by:
\begin{align*}
s^{\max}_\sigma=1-\max\{\chi_c(F)\}=&
\begin{cases}
-\max\{-1+\chi_c(B_0)\}&\textrm{if $m$ is odd and $a+b$ is even}\\
\min\{\chi_c(B_+)-\chi_c(B_-)\}&\textrm{if $m$ is even and $a+b$ is odd}
\end{cases}
\\
s^{\min}_\sigma=1-\min\{\chi_c(F)\}=&
\begin{cases}
-\min\{-1+\chi_c(B_0)\}&\textrm{if $m$ is odd and $a+b$ is even}\\
\max\{\chi_c(B_+)-\chi_c(B_-)\}&\textrm{if $m$ is even and $a+b$ is odd}
\end{cases}
\end{align*}
\end{defn}
Now let us apply this machinery to $A_k$ and $D_k$ singularities.
\subsection{$A_k$ singularities}\label{Morin}
We set $n=m=1$, $j=0$ and: 
$$
g_c(x)=g(x;c)=x^{k+1}+c_1x^{k-1}+\cdots+c_{k-2}x^2+c_{k-1}x. 
$$
Then we have $\chi_c(B_0)=\#\{x:g_c(x)=\varepsilon\}$ and: 
$$
\chi_c(B_+)-\chi_c(B_-)=
\begin{cases}
0&k \  \textrm{even},\\
-1&k \  \textrm{odd}.
\end{cases}
$$
If $k$ is even, then we obtain: 
$$
1-(-1)^{a+b}\chi_c(F)=
\begin{cases}
1-\#\{x\in B:g_c(x)=\varepsilon\}&
\textrm{$a$  even, $b$  even }\\
0
&\textrm{$a$  even, $b$  odd \ }\\
0
&\textrm{$a$  odd,  $b$  even }\\
\#\{x\in B:g_c(x)=\varepsilon\}-1
&\textrm{$a$ odd, $b$ odd \ }
\end{cases}
$$
If $k$ is odd, then we obtain: 
$$
1-(-1)^{a+b}\chi_c(F)=
\begin{cases}
1-\#\{x\in B:g_c(x)=\varepsilon\}&
\textrm{$a$ even, $b$ even }\\
-1
&\textrm{$a$ even, $b$ odd \ }\\
1
&\textrm{$a$ odd, $b$ even }\\
\#\{x\in B:g_c(x)=\varepsilon\}-1
&\textrm{$a$ odd,  $b$  odd \ }
\end{cases}
$$
\subsection{Unfoldings of functions
$(x_1,x_2,z)\mapsto g(x,0)+Q(z)$}
We set $n=m=2$ and $j=0$. 
We consider the map defined by: 
\begin{multline}\label{NFF}
(\BB{R}^{2+a+b+h},0)\to(\BB{R}^{1+h},0),\qquad
(x_1,x_2,z_1,\dots,z_{a+b},c_1,\dots,c_h)\mapsto\\
(g(x_1,x_2,c_1,\dots,c_h)
+{z_1}^2+\cdots+{z_a}^2-{z_{a+1}}^2-\cdots-{z_{a+b}}^2,c_1,\dots,c_h).
\end{multline}
Let $r$ denote the number of branches of the curve defined by $g(x;0)=0$. 
Since $\chi_c(B_0)=-r$, we obtain that:
$$
1+(-1)^{a+b}\chi_c(F)=
\begin{cases}
1-r
&\textrm{$a$ even, $b$  even }\\
\chi_c(B_-)-\chi_c(B_+)
&\textrm{$a$ even, $b$  odd \ }\\
\chi_c(B_+)-\chi_c(B_-)
&\textrm{$a$ odd, \ $b$ even }\\
r-1
&\textrm{$a$ odd, \ $b$ odd \ }
\end{cases}
$$
O.~Viro \cite{Viro1990} described the list of possible smoothings of 
$D_k$ ($k\ge4$), $E_6$, $E_7$, $E_8$, $J_{10}$ 
and non-degenerate $r$-fold points. 
In next subsection, we use this list 
to compute $\chi_c(B_+)-\chi_c(B_-)$ 
for $D_k$ singularities. 
We leave to the reader the computation in the other cases. 
\subsection{$D_k$ singularities}
We denote by $D_k^\pm$ the singularity defined by \eqref{NFF}
with:
$$
g(x;c)=x_1({x_1}^{k-2}\pm {x_2}^2)+c_1x_1+\cdots
+c_{k-2}{x_1}^{k-2}+c_{k-1}x_2. 
$$
\underline{First case:}  $k$ is even and 
$\{x \in \mathbb{R}^2  :   g(x,0)=0 \}$ has 3 branches.

The zero set of $g(x,0)$ looks like  the following:
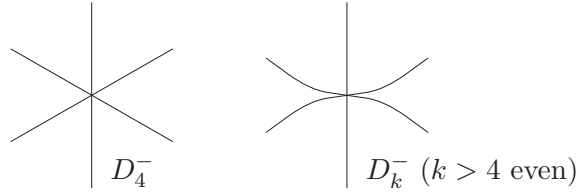
\begin{center}
\setlength{\unitlength}{0.7pt}
\begin{picture}(100,100)(-50,-50)
\path(0,-50)(0,50)
\path(43.3,25)(-43.3,-25)
\path(-43.3,25)(43.3,-25)
\put(10,-45){$D_4^-$}
\end{picture}
\qquad
\begin{picture}(100,100)(-50,-50)
\path(0,-50)(0,50)
\spline(43.3,20)(20,3)(0,0)
\spline(-43.3,20)(-20,3)(0,0)
\spline(43.3,-20)(20,-3)(0,0)
\spline(-43.3,-20)(-20,-3)(0,0)
\put(10,-45){$D_k^-$ ($k>4$ even)}
\end{picture}
\end{center}
First consider the smoothing described by the following picture:
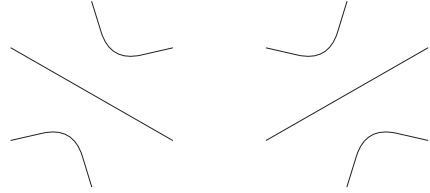
\begin{center}
\setlength{\unitlength}{0.7pt}
\begin{picture}(100,100)(-50,-50)
\spline(0,-50)(-10,-17.3)(-43.3,-25)
\spline(43.3,25)(10,17.3)(0,50)
\path(-43.3,25)(43.3,-25)
\end{picture}
\qquad
\begin{picture}(100,100)(-50,-50)
\spline(0,50)(-10,17.3)(-43.3,25)
\spline(43.3,-25)(10,-17.3)(0,-50)
\path(-43.3,-25)(43.3,25)
\end{picture}
\end{center}
For such a smoothing, it is easy to see $\chi_c(B_+)-\chi_c(B_-)=0$. 

Next we consider the smoothings described by the following pictures: 
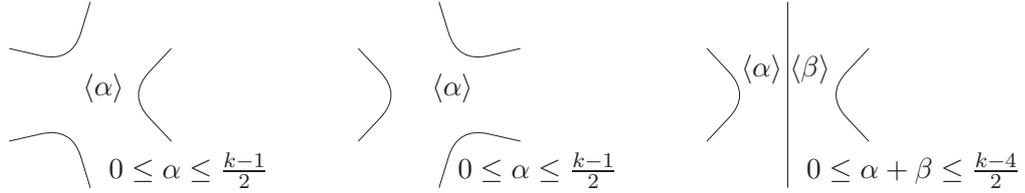
\begin{center}
\setlength{\unitlength}{0.7pt}
\begin{picture}(150,100)(-50,-50)
\spline(0,-50)(-10,-17.3)(-43.3,-25)
\spline(43.3,25)(20,0)(43.3,-25)
\spline(-43.3,25)(-10,17.3)(0,50)
\put(-4,0){$\langle\alpha\rangle$}
\put(10,-45){$0\le\alpha\le\frac{k-1}2$}
\end{picture}
\qquad
\begin{picture}(150,100)(-50,-50)
\spline(0,50)(10,17.3)(43.3,25)
\spline(-43.3,-25)(-20,0)(-43.3,25)
\spline(43.3,-25)(10,-17.3)(0,-50)
\put(-4,0){$\langle\alpha\rangle$}
\put(10,-45){$0\le\alpha\le\frac{k-1}2$}
\end{picture}
\qquad
\begin{picture}(140,100)(-50,-50)
\path(0,50)(0,-50)
\spline(43.3,-25)(20,0)(43.3,25)
\spline(-43.3,-25)(-20,0)(-43.3,25)
\put(-25,10){$\langle\alpha\rangle$}
\put(1,10){$\langle\beta\rangle$}
\put(10,-45){$0\le\alpha+\beta\le\frac{k-4}2$}
\end{picture}
\end{center}
Here $\langle \alpha \rangle$ represents a group of $\alpha$ ovals
without nests. For such smoothings, we see that:
$
\chi_c(B_+)-\chi_c(B_-)=2(1+\alpha), \ 
-2(1+\alpha), \ 2(\alpha-\beta)$ 
respectively.
Then we obtain: 
$$
\chi_c(B_+)-\chi_c(B_-)=-k, -k+2,\dots, k-2, k.
$$
\underline{Second case:}  $k$ is even and $\{x \in \mathbb{R}^2  :    g(x,0)=0 \}$ has 1 branch.

The smoothings are described by the figure on the right-hand side. 
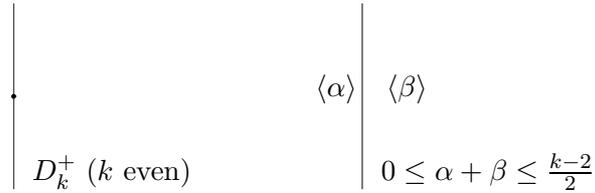
\begin{center}
\setlength{\unitlength}{0.7pt}
\begin{picture}(150,100)(-50,-50)
\path(0,-50)(0,50)
\put(0,0){\circle*{2}}
\put(10,-45){$D_k^+$ ($k$ even)}
\end{picture}
\qquad
\begin{picture}(140,100)(-50,-50)
\path(0,-50)(0,50)
\put(-25,0){$\langle\alpha\rangle$}
\put(13,0){$\langle\beta\rangle$}
\put(10,-45){$0\le\alpha+\beta\le \frac{k-2}2$}
\end{picture}
\end{center}
For such smoothings, we see that
$\chi_c(B_+)-\chi_c(B_-)=2(\alpha-\beta)$. 
Thus we have:
$$
\chi_c(B_+)-\chi_c(B_-)=2-k,4-k,\dots,k-4,k-2.
$$
\underline{Third case:}  $k$ is odd.
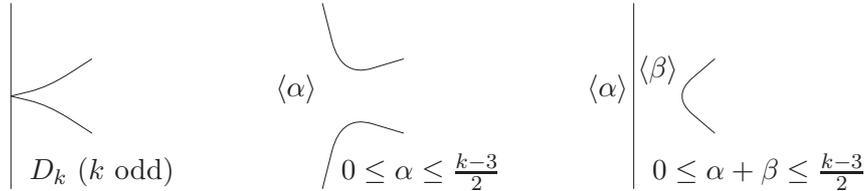
\begin{center}
\setlength{\unitlength}{0.7pt}
\begin{picture}(130,100)(-50,-50)
\path(0,-50)(0,50)
\spline(43.3,20)(20,5)(0,0)
\spline(43.3,-20)(20,-5)(0,0)
\put(10,-45){$D_k$ ($k$ odd)}
\end{picture}
\qquad
\begin{picture}(130,100)(-50,-50)
\put(-25,0){$\langle\alpha\rangle$}
\spline(0,-50)(10,-10)(43.3,-20)
\spline(0,50)(10,10)(43.3,20)
\put(10,-45){$0\le\alpha\le\frac{k-3}2$}
\end{picture}
\qquad
\begin{picture}(130,100)(-50,-50)
\put(-25,0){$\langle\alpha\rangle$}
\put(2,10){$\langle\beta\rangle$}
\path(0,-50)(0,50)
\spline(43.3,-20)(20,0)(43.3,20)
\put(10,-45){$0\le\alpha+\beta\le\frac{k-3}2$}
\end{picture}
\end{center}
For such smoothings, we see that:
$\chi_c(B_+)-\chi_c(B_-)=-1-2\alpha,\ 1-2(\alpha-\beta)$, 
respectively.
Thus we have:
$$
\chi_c(B_+)-\chi_c(B_-)=2-k,4-k,\dots,k-4,k-2.
$$

\lsection{Study of stable maps $f:M\to N$ with $\dim M \ge \dim N$}
\label{SecSatble1}
Let $f:M\to N$ be a stable map between two smooth manifolds $M$ and $N$.
Let $m=\dim M$ and $n=\dim N$.
We assume that $m \ge n$, that $N$ is connected and that $M$ and $N$ have finite topological types.
Let $\sigma$ denote the singularity type given by the genotype : $x\mapsto(g(x;0),g'(x;0))$ in the notation 
of \eqref{PreNormalForm}.   
Then the genotype $\sigma$ gives rise to two singularity types of $f$ :  we say that $f$ is of type $\sigma^+$ (resp. $\sigma^-$) 
if, in the expression of given \eqref{PreNormalForm}, 
 $b$ is even (resp. odd).
The definition of $\sigma^+$ and $\sigma^-$ is ad hoc, 
since it depends on the normal form \eqref{PreNormalForm}. 
It seems to be no natural way to define the sign in general.  
We set:
$$
\sigma^\pm(f)=\{x\in M: f_x \textrm{ has singularity of type $\sigma^\pm$}\},
$$
where $f_x:(M,x)\to(N,f(x))$ is the germ of $f$ at $x$. 
Let $\Sigma(f)$ denote the critical set of $f$. 

Since $f$ is stable, $\Sigma(f)\cap f^{-1}(y)$ is a finite set for each 
$y\in N$.  
Then $f$ defines a multi-germ:
$$
f_y:(M,\Sigma(f)\cap f^{-1}(y))\to(N,y). 
$$
Let $\tau$ denote a type of singularities of stable multi-germs and: 
$$
N_\tau(f)=\{y\in N: f_y \textrm{ has singularities of type $\tau$}\}.
$$
\subsection{Case $m-n$ is odd}
If $m-n$ is odd, then $\chi_c(f^{-1}(y')\cap \overline{B_\varepsilon(x)})$
does not depend on the choice of regular value $y'$ nearby $f(x)$, where $B_\varepsilon(x)$ denotes the open ball of small radius $\varepsilon$ centered at $x$ in $M$.  
Indeed, $f^{-1}(y') \cap \overline{B_\varepsilon(x)}$ is a compact odd-dimensional manifold with boundary and so:
$$\chi_c(f^{-1}(y') \cap  \overline{B_\varepsilon(x)})=\chi(f^{-1}(y') \cap  \overline{B_\varepsilon(x)})=
\frac{1}{2} \chi(f^{-1}(y') \cap \partial  \overline{B_\varepsilon(x)}).$$
But the last Euler characteristic is equal to $\chi(f^{-1}(f(x)) \cap \partial \overline{B_\varepsilon(x)})$. If $x$ is of type $\nu$ then we 
denote by $c_\nu$ the Euler characteristic $\chi_c(f^{-1}(y') \cap \overline{B_\varepsilon(x)})$. 

Replacing the ball of small radius with a ball with big radius and assuming that $f$ is locally trivial at infinity, we may establish in a similar way that $\chi_c(f^{-1}(y))$ does not depend on the choice of the regular value $y$ of $f$.
We denote this Euler characteristic by $\chi_f$. 
\begin{thm}\label{RThm1}
Assume that $f:M\to N$ is locally trivial at infinity and
has finitely many singularity types (this is the case when $(m,n)$ is
a pair of nice dimensions  in Mather's sense). 
Then we have:
\begin{equation}\label{F1}
\medsum_{\nu}c_\nu\chi_c(\nu(f))=\chi_f\chi_c(N),
\end{equation}
provided that the $\chi_c(\nu(f))$'s and $\chi_f$ are finite.
Moreover, if all singularities of $f$ are versal unfoldings of function-germs
then we have: 
\begin{equation}\label{F2}
\chi_c(M)-\chi_f\chi_c(N)=\medsum_\sigma s_\sigma
\bigl[\chi_c(\sigma^+(f))-\chi_c(\sigma^-(f))\bigr],
\end{equation}
where $\sigma$ denotes the singularity type of the genotype and 
$s_\sigma$ is defined as in Definition \ref{DefSsigma}.
\end{thm}
\begin{proof}
We consider the stratification of $f$ defined by the types of 
singularities (see Nakai's paper \cite[\S1]{Nakai2000}) and we define 
$\mathcal S(M)$, $\mathcal S(N)$ 
as the subset algebras generated by the strata and fibers of $f$.   
Then $(\mathcal S(X),\mathcal S(Y))$ fits to the map $f$. 
Set $\mu_X=\chi_c$, $\mu_Y=\chi_c$ and:
$$
\varphi(x)=\chi_c(f^{-1}(y')\cap \overline{B_\varepsilon(x)}),
$$
where $y'$ is a regular value nearby $f(x)$. 
Applying Corollary \ref{Fubini2} for $\varphi$, 
Lemma \ref{L1} and Remark \ref{L2} below,  we obtain:
\begin{equation}\label{F11}
\medsum_{\nu}c_\nu\chi_c(\nu(f))=\chi_f\chi_c(N).
\end{equation}
By the additivity of the Euler characteristic with compact support, we get:
$$
\chi_c(M)-\chi_f\chi_c(N)
=
\medsum_{\nu}(1-c_\nu)\chi_c(\nu(f)).
$$
If all the singularities are versal unfoldings of function-germs then
 each genotype gives rise to two singularity types $\sigma^+(f)$ and
 $\sigma^-(f)$ and $1-c_{\sigma^-}=-(1-c_{\sigma^+})=-s_\sigma$ by
 Remark  \ref{L2} below and the computations made in Section \ref{Euler
 characteristics of local generic fibers}.

\end{proof}
\begin{lem}\label{L1}
Let $f:M\to N$ be a smooth map such that:
\begin{itemize}
\item $\dim M-\dim N$ is odd,
\item $f|_{\Sigma(f)}$ is finite,
\item $f$ is locally trivial at infinity. 
\end{itemize}
Then for each $y \in N$, we have: 
$$
f_*\varphi(y)
=\int_{f^{-1}(y)}\varphi(x)d\chi_c
=\chi_c(f^{-1}(y')), 
$$
where $y'$ is a regular value of $f$ close to $y$.
\end{lem}
\begin{proof}
Set $\{x_1,\dots,x_s\}=f^{-1}(y)\cap\Sigma(f)$. Take a regular value $y'$ of $f$ near $y$. Then:
\begin{align*}
\chi_c(f^{-1}(y'))
=&\chi_c(f^{-1}(y')\setminus\cup_i\overline{B_\varepsilon(x_i)})
+\medsum_i\chi_c(f^{-1}(y')\cap \overline{B_\varepsilon(x_i)})\\
=&\chi_c(f^{-1}(y)\setminus\cup_i\overline{B_\varepsilon(x_i)})
+\medsum_i\chi_c(f^{-1}(y')\cap \overline{B_\varepsilon(x_i)})\\
=&\chi_c(f^{-1}(y)\setminus\{x_1,\dots,x_s\})
+\medsum_i\varphi(x_i)\\
=&\int_{f^{-1}(y)\setminus\{x_1,\dots,x_s\}}\varphi(x)d\chi_c
+\int_{\{x_1,\dots,x_s\}}\varphi(x)d\chi_c\\
=&\int_{f^{-1}(y)}\varphi(x)d\chi_c.
\end{align*}
\end{proof}
\begin{rem}\label{L2}
Set $\phi(x)=\chi_c(f^{-1}(y')\cap B_\varepsilon(x))$ where $y'$ is a
 regular value nearby $f(x)$.   
Then:
$$
\varphi(x)=\phi(x)+\chi_c(f^{-1}(y')\cap S_\varepsilon(x)), 
$$
where $S_\varepsilon(x)$ is the sphere of radius $\varepsilon$ centered at $x$.
If $f^{-1}(y')\cap \overline{B_\varepsilon(x)}$ is an odd
 dimensional  manifold with boundary $f^{-1}(y')\cap S_\varepsilon(x)$,
then we obtain $\phi(x)=-\varphi(x)$, 
since: 
$$
2\varphi(x)
=\chi_c(f^{-1}(y')\cap S_\varepsilon(x))
=-2\phi(x).
$$
Similarly 
if $f^{-1}(y')\cap \overline{B_\varepsilon(x)}$ is an even
dimensional  manifold with boundary $f^{-1}(y')\cap S_\varepsilon(x)$,
we obtain that $\phi(x)=\varphi(x)$.
\end{rem}

%

\begin{cor}\label{RMorin1}
Assume that the map $f$ satisfies the assumptions of Theorem \ref{RThm1} and 
has at worst $A_n$ singularities. Then, we have:
$$
\chi_c(M)-\chi_f\, \chi_c(N)
=\medsum_{k :\textrm{odd}}
\bigl[\chi_c( A_k^+ (f))
-\chi_c( A_k^-(f))\bigr].
$$
\end{cor}
\begin{proof}
Using the computations in
Section \ref{Euler characteristics of local generic fibers}, 
we see that $s_{A_k}=0$ is $k$ is even and $s_{A_k}=1$ if $k$ is odd. 
\end{proof}
\begin{cor}\label{RMorin11}
Assume that the map $f$ satisfies the assumptions of Theorem \ref{RThm1} and 
has only stable singularities locally defined by \eqref{NFF}.
We denote $\sigma_r$ the union of singularities types
so that the number of branches of $g(x_1,x_2;0)=0$ near $0$ is $r$.  
We denote by $\sigma_r^+$ (resp. $\sigma_r^-$) 
the union of such singularities types with even (resp. odd) $b$. 
Then, we have: 
$$
\chi_c(M)-\chi_f\, \chi_c(N)
=
\medsum_{r}
(1-r)\bigl[\chi_c( \sigma_r^+ (f))
-\chi_c(\sigma_r^-(f))\bigr].
$$
\end{cor}
\begin{proof}
Using the computations in
Section \ref{Euler characteristics of local generic fibers}, 
we see that $s_{ \sigma_r}=1-r$ . 
\end{proof}

\subsection{Case $m-n$ is even and $m-n >0$}
If $m-n$ is even and non-zero then $\chi_c(f^{-1}(y') \cap
\overline{B_\varepsilon(x)})$ depends on the choice of the regular value
$y'$ nearby $f(x)$. But its parity does not depend on $y'$.
Indeed, $f^{-1}(y') \cap \overline{B_\varepsilon(x)}$ is a compact 
even-dimensional manifold with boundary and so:
\begin{align*}
\chi_c(f^{-1}(y') \cap  \overline{B_\varepsilon(x)})
\equiv&\chi(f^{-1}(y')\cap  \overline{B_\varepsilon(x)})\\
\equiv&\psi (f^{-1}(y')\cap S_\varepsilon(x))\\
\equiv&\psi (f^{-1}(f(x))\cap S_\varepsilon(x)) \pmod 2,
\end{align*}
where $\psi$ denotes the semi-characteristic, 
i.e., half the sum of the mod $2$ Betti numbers (see \cite{Wall1983}).

If a point $x$ in $M$ is of singular type $\nu$ then we denote by
$c_\nu$ the mod $2$ Euler characteristic $\chi_c(f^{-1}(y') \cap
\overline{B_\varepsilon(x)})$. We will denote by $\chi_f$ the mod $2$
Euler characteristic $\chi_c(f^{-1}(y))$ where $y$ is a regular value of
$f$. The following theorem is proved in the same way as Theorem
\ref{RThm1}.
\begin{thm}\label{RThm1B}
Assume that $f:M\to N$ is locally trivial at infinity and
has finitely many singularity types (this is the case when $(m,n)$ is
a pair of nice dimensions in Mather's sense). 
Then we have:
\begin{equation}\label{F111}
\medsum_{\nu}c_\nu\chi_c(\nu(f))\equiv\chi_f\chi_c(N) \pmod{2},
\end{equation}
provided that the $\chi_c(\nu(f))$'s and $\chi_f$ are finite.
Moreover, if all singularities of $f$ are versal unfoldings of function-germs
then we have: 
\begin{equation}\label{F22}
\chi_c(M)-\chi_f\chi_c(N)\equiv\medsum_\sigma s_\sigma
\bigl[\chi_c(\sigma^+(f))-\chi_c(\sigma^-(f))\bigr] \pmod{2},
\end{equation}
where $\sigma$ denotes the singularity type of the genotype and
 $s_\sigma$ is defined as in Definition \ref{DefSsigma}.
\end{thm}
This theorem gives a mod $2$ equality. Nevertheless, it is still
possible to find integral relations between the topology 
of the source, the target and the singular set.
 
Let $\nu$ denote a singularity type of a map-germ. 
Let $c_\nu^{\max}$ (resp. $c_\nu^{\min}$) denote 
the maximal (resp. minimum) of all possible Euler characteristics of 
local regular fibers nearby the singular fiber. 
Set also:
\begin{align*}
N_j^{\max}=&\big\{y\in N  : 
j=\max\{\chi_c(f^{-1}(y')):y' \textrm{ a regular value nearby } y\} \big\},\\
N_j^{\min}=&\big\{y\in N  :  
j=\min\{\chi_c(f^{-1}(y')):y' \textrm{ a regular value nearby } y\} \big\}.
\end{align*}
\begin{thm}\label{RThm1C}
If a smooth map $f:M\to N$ is locally trivial at infinity and
has finitely many singularity types (this is the case if $(m,n)$ is
a pair of nice dimensions in  Mather's sense), 
then: 
\begin{align*}
\medsum_{\nu}c_\nu^{\max}\cdot\chi_c(\nu(f))
\ge&\medsum_jj\chi_c(N_j^{\max}), \\
\medsum_{\nu}c_\nu^{\min}\cdot\chi_c(\nu(f))
\le&\medsum_jj\chi_c(N_j^{\min}),
\end{align*}
provided the $\chi_c(\nu(f))$'s, the $\chi_c(N_j^{\max})$'s  and the $\chi_c(N_j^{\min})$'s are finite.
If $f$ is stable, we have the equalities. 
Moreover, if all singularities are versal unfoldings of function-germs
then: 
\begin{align*}
\chi_c(M)-\medsum_jj\chi_c(N^{\max}_j)
=&\sum_\sigma 
\bigl[s_\sigma^{\max}\chi_c(\sigma^+(f))
-s_\sigma^{\min}\chi_c(\sigma^-(f))\bigr],\\
\chi_c(M)-\medsum_jj\chi_c(N^{\min}_j)
=&\sum_\sigma 
\bigl[s_\sigma^{\min}\chi_c(\sigma^+(f))
-s_{\sigma}^{\max}\chi_c(\sigma^-(f))\bigr],
\end{align*}
where $\sigma$ denotes the singularity type of the genotype and
 $s_\sigma^{\max}$ and $s_\sigma^{\min}$ are defined 
in Definition \ref{DefSsigma}.
\end{thm}
\begin{proof}
To get the first inequalities, we apply the same method as we did in the proof of Theorem \ref{RThm1} with the two following constructible functions $\varphi_{\max}$ and $\varphi_{\min}$:
\begin{align*}
\varphi_{\max}(x)=&\max\{\chi_c(f^{-1}(y')\cap B_\varepsilon(x)) :
y' \textrm{ is regular value nearby $f(x)$}\},\\
\varphi_{\min}(x)=&\min\{\chi_c(f^{-1}(y')\cap B_\varepsilon(x)) :
y' \textrm{ is regular value nearby $f(x)$}\}. 
\end{align*}
We also use Lemma \ref{L3} below.

When $f$ is stable, by the additivity of the Euler characteristic with compact support, we get:
$$\chi_c(M)-\sum_j \chi_c(N_j^{\max})=\sum_{\nu}(1-c_\nu^{\max}) \chi_c(\nu(f)),$$
$$\chi_c(M)-\sum_j \chi_c(N_j^{\min})=\sum_{\nu}(1-c_\nu^{\min}) \chi_c(\nu(f)).$$
If all the singularities are versal unfoldings of function-germs then
 each genotype gives rise to two singularity types $\sigma_+(f)$ and
 $\sigma_-(f)$. Using the computations done in 
Section \ref{Euler characteristics of local generic fibers}, we see that:
$$
1-c_{\sigma^+(f)}^{\max}=s_\sigma^{\max}, \ 
1-c_{\sigma^-(f)}^{\max}=-s_\sigma^{\min}, \ 
1-c_{\sigma^+(f)}^{\min}=s_\sigma^{\min}, \hbox{ and } 
1-c_{\sigma^-(f)}^{\min}=-s_\sigma^{\max}.
$$
\end{proof}
\begin{lem}\label{L3}
Let $f:M\to N$ be a smooth map such that:
\begin{itemize}
\item $\dim M-\dim N$ is even,
\item $f|_{\Sigma(f)}$ is finite, 
\item $f$ is locally trivial at infinity.
\end{itemize}
Then we have: 
\begin{align*}
f_*\varphi_{\max}(y)\ge&
\max\{\chi_c(f^{-1}(y')):y' \textrm{ a regular value
 nearby $y$}\},
\\
f_*\varphi_{\min}(y)\le&
\min\{\chi_c(f^{-1}(y')):y' \textrm{ a regular value
 nearby $y$}\}.
\end{align*}
We have the equalities when $f$ is stable. 
\end{lem}
\begin{proof}
Set $\{x_1,\dots,x_s\}=f^{-1}(y)\cap\Sigma(f)$.
Take a regular value $y'$ of $f$ near $y$. Then:
\begin{align*}
\chi_c(f^{-1}(y'))
=&\chi_c(f^{-1}(y')\setminus\cup_i\overline{B_\varepsilon(x_i)})
+\medsum_i\chi_c(f^{-1}(y')\cap \overline{B_\varepsilon(x_i)})\\
=&\chi_c(f^{-1}(y)\setminus\cup_i\overline{B_\varepsilon(x_i)})
+\medsum_i\chi_c(f^{-1}(y')\cap \overline{B_\varepsilon(x_i)})\\
\le&\chi_c(f^{-1}(y)\setminus\{x_1,\dots,x_s\})
+\medsum_i\varphi_{\max}(x_i)\\
=&\int_{f^{-1}(y)\setminus\{x_1,\dots,x_s\}}\varphi_{\max}(x)d\chi_c
+\int_{\{x_1,\dots,x_s\}}\varphi_{\max}(x)d\chi_c\\
=&\int_{f^{-1}(y)}\varphi_{\max}(x)d\chi_c
=f_*\varphi_{\max}(y).
\end{align*}
When $f$ is stable, we see that the equality is attained by some $y'$ 
using the fact (i)$\Longleftrightarrow$(iii) of \cite[Lemma 1.5]{Wall2009}. 

Similarly we obtain: 
\begin{align*}
\chi_c(f^{-1}(y'))
=&\chi_c(f^{-1}(y')\setminus\cup_i\overline{B_\varepsilon(x_i)})
+\medsum_i\chi_c(f^{-1}(y')\cap \overline{B_\varepsilon(x_i)})\\
=&\chi_c(f^{-1}(y)\setminus\cup_i\overline{B_\varepsilon(x_i)})
+\medsum_i\chi_c(f^{-1}(y')\cap \overline{B_\varepsilon(x_i)})\\
\ge&\chi_c(f^{-1}(y)\setminus\{x_1,\dots,x_s\})
+\medsum_i\varphi_{\min}(x_i)\\
=&\int_{f^{-1}(y)\setminus\{x_1,\dots,x_s\}}\varphi_{\min}(x)d\chi_c
+\int_{\{x_1,\dots,x_s\}}\varphi_{\min}(x)d\chi_c\\
=&\int_{f^{-1}(y)}\varphi_{\min}(x)d\chi_c
=f_*\varphi_{\min}(y). 
\end{align*}
When $f$ is stable, we see that the equality is attained by some $y'$ 
using the fact (i)$\Longleftrightarrow$(iii) of \cite[Lemma 1.5]{Wall2009}. 
\end{proof}
Now let us apply this theorem to the case of a map having at worst $D_n$
singularities. Using the computations in Section \ref{Euler characteristics of local generic fibers}, 
we see that:
$$
1-c_\nu^{\max}
=
\begin{cases}
-k&\textrm{ if } x\in A_k^+(f),\\
-1&\textrm{ if } x\in A_k^-(f),\ k\textrm{  odd},\\
0&\textrm{ if } x\in A_k^-(f),\ k\textrm{  even},\\
k&\textrm{ if } x\in D_k^-(f),\ k\textrm{  even},\\
k-2&\textrm{ if } x\in D_k^+(f),\ k\textrm{  even},\\
k-2&\textrm{ if } x\in D_k(f),\ k\textrm{  odd},
\end{cases}
\qquad
1-c_\nu^{\min}
=
\begin{cases}
k&\textrm{ if } x\in A_k^-(f)\\
1&\textrm{ if } x\in A_k^+(f),\ k\textrm{  odd},\\
0&\textrm{ if } x\in A_k^+(f),\ k\textrm{   even},\\
-k&\textrm{ if } x\in D_k^-(f),\ k\textrm{  even},\\
2-k&\textrm{ if } x\in D_k^+(f),\ k\textrm{  even},\\
2-k&\textrm{ if } x\in D_k(f),\ k\textrm{  odd}.
\end{cases}
$$
\begin{cor}\label{RMorin2}
If the map $f$ satisfies the assumptions of 
Theorem \ref{RThm1C} and has at worst $D_n$ singularities then:
\footnotesize
\begin{align*}
\chi_c(M)-\sum_jj\chi_c(N_j^{\max})
=&
-\medsum_{k>0}k\chi_c(A^+_k(f))
-\medsum_{k: \textrm{odd}}\chi_c(A_k^-(f))
+\sum_{k}(k-2)\chi_c(D_k(f))
+2\sum_{k: \e}\chi_c(D_k^-(f))
\\
\chi_c(M)-\sum_jj\chi_c(N_j^{\min})
=&
\medsum_{k>0}k\chi_c(A_k^-(f))
+\medsum_{k :\textrm{odd}}\chi_c(A_k^+(f))
+\sum_{k}(2-k)\chi_c(D_k(f))
-2\sum_{k: \e}\chi_c(D_k^-(f))
\end{align*}
\end{cor}
\begin{proof}
Combine the previous theorem with
the above expressions of $\varphi_{\max}$ and $\varphi_{\min}$.
\end{proof}

\begin{cor}\label{RMorin3}
Assume that a map $f$ satisfies the assumptions of Theorem \ref{RThm1C}
 and has at worst $A_n$ singularities.\\
When $\dim N=1$, we have:
$$
\sum_jj\chi_c(N_j^{\max})=\chi_c(M)+\chi_c(A_1(f)),\qquad
\sum_jj\chi_c(N_j^{\min})=\chi_c(M)-\chi_c(A_1(f)), 
$$
and thus:
$$
\medsum_{j}\tfrac{j}2[\chi_c(N^{\max}_j)+\chi_c(N^{\min}_j)]=
\chi_c(M),
$$
$$
\medsum_{j}\tfrac{j}2[\chi_c(N^{\max}_j)-\chi_c(N^{\min}_j)]=
\chi_c(A_1(f))=\chi_c(\Sigma(f)).
$$
When $\dim N=2$, we have:
\begin{align*}
\sum_jj\chi_c(N_j^{\max})
=&\chi_c(M)
+\chi_c(A_1(f))+2\#(A^+_2(f)),\\
\sum_jj\chi_c(N_j^{\min})
=&\chi_c(M)
-\chi_c(A_1(f))-2\#(A_2^-(f)),
\end{align*}
and thus:
$$
\medsum_{j}\tfrac{j}2[\chi_c(N^{\max}_j)+\chi_c(N^{\min}_j)]=
\chi_c(M)+\#(A^+_2(f))-\#(A^-_2(f)),
$$
$$
\medsum_{j}\tfrac{j}2[\chi_c(N^{\max}_j)-\chi_c(N^{\min}_j)]=
\chi_c(A_1(f))+\#(A_2(f))=\chi_c(\Sigma(f)).
$$
When $\dim N=3$, we have:
\begin{align*}
\sum_jj\chi_c(N_j^{\max})
=&\chi_c(M)
+\chi_c(A_1(f))+2\chi_c(A^+_2(f))+\#(A_3(f))+2\#(A_3^+(f)),\\
\sum_jj\chi_c(N_j^{\min})
=&\chi_c(M)
-\chi_c(A_1(f))-2\chi_c(A_2^-(f))-\#(A_3(f))-2\#(A_3^-(f)), 
\end{align*}
and thus:
$$
\medsum_{j}\tfrac{j}2[\chi_c(N^{\max}_j)+\chi_c(N^{\min}_j)]
=\chi_c(M)+\chi_c(A^+_2(f))-\chi_c(A^-_2(f))+\#(A^+_3(f))-\#(A^-_3(f)),
$$
$$
\medsum_{j}\tfrac{j}2[\chi_c(N^{\max}_j)-\chi_c(N^{\min}_j)]
=\chi_c(A_1(f))+\chi_c(A_2(f))+2\#(A_3(f))
=\chi_c(\Sigma(f))+\#(A_3(f)).
$$
\end{cor}
\subsection{Case $m-n=0$}
Here we assume that $M$ and $N$ are oriented and have the same dimension
$n$. If a point $x$ in $M$ is of type $\nu$, we denote by $d_\nu$ the
local topological degree of the map-germ 
$f:(M,x)\to(N,f(x))$. We assume that $f$ is finite-to-one and that $f$
is locally trivial at infinity. In this situation, it is possible to
define the mapping degree of $f$ as follows:
$$
\deg f =\sum_{x \in f^{-1}(y)} \deg \left( f: (M,x) \to (N,f(x))
\right),
$$
where $y$ is a regular value of $f$.
\begin{thm}\label{RThm1D}
Assume that a map $f:M\to N$ is finite-to-one, locally trivial at
 infinity and has finitely many singularity types. We also assume that
 $M$ and $N$ are oriented and that $N$ is connected. Then:
$$
\medsum_\nu d_\sigma\chi_c(\nu(f))
=(\deg f)\chi_c(N),
$$
provided that  the $\chi_c(\nu(f))$'s are finite.
\end{thm}
\begin{proof}
We consider the stratification of $f$ defined by the types of 
singularities (see Nakai's paper \cite[\S1]{Nakai2000}) and we define 
$\mathcal S(M)$, $\mathcal S(N)$ 
as the subset algebras generated by the strata and fibers of $f$.   
Then $(\mathcal S(X),\mathcal S(Y))$ fits to the map $f$. 
Set $\mu_X=\chi_c$, $\mu_Y=\chi_c$ and:
$$
\varphi(x)= \deg \left( f: (M,x) \to (N,f(x)) \right).
$$
Applying Corollary \ref{Fubini2} for $\varphi$ and remarking that $f_*\varphi(y)=\deg f$, we obtain the result.
\end{proof}
If $x$ is a point of type $A_k$ with $k$ even, we say that $x$ belongs to $A_k^+(f)$ (resp. $A_k^-(f)$) if $\deg\{f:(M,x)\to(N,f(x))\}=1$ (resp. $-1$).

\begin{cor}\label{RMorin4}
Assume that $f$ satisfies the assumptions of Theorem \ref{RThm1D} and has at worst $A_n$ singularities. Then we have:
$$
\medsum_{k: \e}
\bigl[\chi_c(A_k^+(f))-\chi_c(A_k^-(f))\bigr]
=(\deg f)\chi_c(N).
$$
\end{cor}
\begin{proof}
Apply the previous theorem and the fact that 
$\deg\{f:(M,x)\to(N,f(x))\}=0$ if $x \in A_k(f)$, $k$ odd.
\end{proof}

The map $f:(\BB{R}^4,0)\to(\BB{R}^4,0)$ is an $I_{2,2}^\pm$ singularity  if $f$ is defined by:
$$
(x,y,a,b)\mapsto(x^2\pm y^2+ax+by,xy,a,b). 
$$
This is the only singularity of stable-germs which is not a Morin
singularity from $\BB{R}^4$ to $\BB{R}^4$. We can state:
\begin{cor}\label{RMorin5}
Assume that $f$ satisfies the assumptions of Theroem \ref{RThm1D} and that $n=4$. Then we have:
$$
\medsum_{k: \e}
\bigl[\chi_c(A_k^+(f))-\chi_c(A_k^-(f))\bigr]
+2\#(I_{2,2}^-(f))=(\deg f) \chi_c(N).
$$
\end{cor}
\begin{proof}
Remark that the mapping degree of  $f_x$ is 2 (resp. 0) when $x$ is an
$I_{2,2}^-$ ($I_{2,2}^+$) point. 
\end{proof}
A similar discussion shows the following:
\begin{thm}
Assume that a map $f:M\to N$ is finite-to-one, locally trivial at infinity and has finitely many singularity types. We assume that $M$ or $N$ may not be orientable and that $N$ is connected.
Then:
$$
\medsum_\sigma d_\sigma\chi_c(M_\sigma(f))
\equiv(\deg f)\chi_c(N)\pmod{2}.
$$
\end{thm}

\lsection{Applications to Morin maps}\label{SecMorin}
In this section, we apply the results of the previous section to Morin
maps. We will consider three different settings : Morin maps
from a compact manifold $M$ to a connected manifold $N$ such that $\dim
M-\dim N$ is odd, Morin maps from a compact manifold $M$ to a
connected manifold $N$ with $\dim M=\dim N$, Morin perturbations of
smooth map-germs.

\subsection{Morin maps from $M^m$ to $N^n$, $m-n$ odd}
Let $f : M^m \to N^n$ be a Morin mapping from a compact $m$-dimensional
manifold $M$ to a connected $n$-dimensional manifold $N$.

Let us recall that a point $p$ in $M$ is of type $A_k$ if its genotype
is $x^{k+1}$. This means that  there exist a local
coordinate sytem $(x_1,\ldots,x_m)$ centered at $p$ and a local
coordinate system $(y_1,\ldots,y_n)$ centered at $f(p)$ such that $f$
has the following normal form:
$$ \begin{array}{l}
y_i \circ f = x_i \hbox{ for } i \le n-1 ,\cr
y_n \circ f = x_n^{k+1} + \sum_{i=1}^{k-1} x_ix_n^{k-i} +x_{n+1}^2 +\cdots+x_{n+\lambda-1}^2-x_{n+\lambda}^2-\cdots-x_m^2.
\cr 
\end{array} $$
Note that $x$ belongs to $A_k^+(f)$ (resp. $A_k^-(f)$) if and only if $m-n-\lambda+1$ is even (resp. odd). We should remark also that if $k$ is odd then $x \in A_k^+(f)$ (resp. $A_k^-(f)$) if and only if
$\chi_c(f^{-1}(y') \cap \overline{B_\varepsilon(x)})=\chi(f^{-1}(y') \cap \overline{B_\varepsilon(x)})=0$ (resp. $2$) where $y'$ is a regular value of $f$ close to $f(x)$. It is well known that for $k\ge 1$, the $A_k(f)$'s are smooth manifolds of dimension $n-k$, that the $\overline{A_k(f)}$'s are smooth manifolds with boundary and that:
$$\overline{A_k(f)}= \cup_{i \ge k} A_i(f) , \  \ \partial{\overline{A_k(f)}}=\cup_{i>k} A_i(f).$$
We will describe more precisely the structure of the $A_k^{\pm}(f)$'s. 
\begin{prop}\label{RProp6A}
If $k$ is odd then $\overline{A_k^+(f)}$ and $\overline{A_k^-(f)}$ are compact manifolds with boundary of dimension $n-k$.
Furthermore $\partial \overline{A_k^+(f)}=\partial \overline{A_k^-(f)}= \overline{A_{k+1}(f)}$.
\end{prop}
\begin{proof} 
Let $p$ be a point in $A_k(f)$, $k$ odd. There exist local coordinates around $p$ and $f(p)$ such that $f$
has the form:
$$ \begin{array}{l}
y_i \circ f = x_i \hbox{ for } i \le n-1 ,\cr
y_n \circ f = x_n^{k+1} + \sum_{i=1}^{k-1} x_ix_n^{k-i} +x_{n+1}^2 +\cdots+x_{n+\lambda-1}^2-x_{n+\lambda}^2-\cdots-x_m^2. 
\end{array}$$
Let us write $g=y_n \circ f$. Around $p$, $A_k(f)$ is defined by $\frac{\partial g}{\partial x_n}=\cdots=\frac{\partial^k
g}{\partial x_n^k}=0$ and $x_{n+1}=\cdots=x_m=0$. 
It is easy to see that this is equivalent to $x_1=\cdots=x_{k-1}=0$ and
$x_n=\cdots=x_m=0$. This proves that $A_k(f)$ is a manifold  of dimension $n-k$. Let $q=(q_1,\ldots,q_m) \in A_k(f)$ be a point
close to $p$. We have $q_1=\ldots=q_{k-1}=0$ and $q_n=\ldots=q_m=0$. For $i\in \{k,\ldots,n-1\}$, let us put
$z_i=x_i-q_i$ and $w_i=y_i-q_i$. For $i \notin \{k,\ldots,n-1\}$, let us put $z_i=x_i$ and $w_i=y_i$. Then
$(z_1,\ldots,z_m)$ and $(w_1,\ldots,w_n)$ are local coordinate systems centered at $q$ and $f(q)$. In these systems, $f$
has the form:
$$\begin{array}{l}
w_i \circ f = z_i \hbox{ for } i \le n-1, \cr
w_n \circ f = z_n^{k+1} + \sum_{i=1}^{k-1} z_iz_n^{k-i} +z_{n+1}^2 +\cdots+z_{n+\lambda-1}^2-z_{n+\lambda}^2-\cdots-z_m^2.
\cr 
\end{array} $$
We conclude that $q $ belongs to $A_k^+(f)$ (resp. $A_k^-(f)$) if and only if $p$ belongs to $A_k^+(f)$ (resp.
$A_k^-(f)$). This proves that the sets  $A_k^+(f)$ and $A_k^-(f)$ are open subsets of $A_k(f)$, hence manifolds of dimension $n-k$. 

We know that $\overline{A_k(f)}=\cup_{l\ge k} A_l(f)$. Let $l>k$ and let $p \in A_l(f)$. There are local coordinates
systems around $p$ and $f(p)$ such that $f$ has the form:
$$\begin{array}{l}
y_i \circ f = x_i \hbox{ for } i \le n-1, \cr
y_n \circ f = x_n^{l+1} + \sum_{i=1}^{l-1} x_ix_n^{l-i} +x_{n+1}^2 +\cdots+x_{n+\lambda-1}^2-x_{n+\lambda}^2-\cdots-x_m^2.
\cr 
\end{array}$$
Let us denote by $g$ the function $y_n \circ f$. We have:
$$A_k(f)= \left\{ \frac{\partial g}{\partial x_n}=\cdots=\frac{\partial^k g}{\partial x_n^k}=0,x_{n+1}=\cdots=x_m=0,\frac{\partial^{k+1} g}{\partial x_n^{k+1}}\not=0
\right\},$$
and
$$\overline{A_{k+1}(f)}= \left\{ \frac{\partial g}{\partial x_n}=\cdots=\frac{\partial^{k+1} g}{\partial x_n^{k+1}}=0,x_{n+1}=\cdots=x_m=0
\right\}.$$
Let $q=(q_1,\ldots,q_n,0,\ldots,0)$ be a point in $A_k(f)$ close to $p$. Let us find when $q \in A_k^+(f)$ or $q\in
A_k^-(f)$. For this we have to compute $\varphi(q)=\chi(f^{-1}(y') \cap \overline{B_\varepsilon(q)})$ where $y'$ is a regular value of $f$ close to $f(q)$. 
Since it does not depend on the choice of the regular value because $m-n$ is odd, let us compute $\chi(f^{-1}(\tilde{y}) \cap \overline{B_\varepsilon(q)})$ where $\tilde{y}=(q_1,\ldots,q_{n-1},
q_n+\epsilon)$ and $\epsilon$ is a small real number.
So we have to look for the
zeros lying close to $q$ of the following system:
$$\left\{ \begin{array}{l}
y_i \circ f (x)= q_i \hbox{ for } i \le n-1 \cr
g(x)=g(q)+ \epsilon .\cr
\end{array} \right.$$
This system is equivalent to:
$$\left\{ \begin{array}{l}
x_i= q_i \hbox{ for } i \le n-1 \cr
g(q_1,\ldots,q_{n-1},q_n+x_n',x_{n+1},\ldots,x_m)=g(q)+ \epsilon .\cr
\end{array} \right.$$
But we have:
\begin{align*}
&g(q_1,\ldots,q_{n-1},q_n+x_n',x_{n+1},\ldots,x_m)\\
=&g(q_1,\ldots,q_{n-1},q_n+x_n',0,\ldots,0) 
+x_{n+1}^2+\cdots+x_{n+\lambda-1}^2-x_{n+\lambda}^2-\cdots-x_m^2\\
=&g(q)+\sum_{i\ge k+1} \frac{1}{i!}\frac{\partial^i g}{\partial
x_n^i}(q){x_n'}^i+x_{n+1}^2+\cdots
+x_{n+\lambda-1}^2-x_{n+\lambda}^2-\cdots-x_m^2 
\\
=& 
g(q)+g'(x'_n,x_{n+1},\ldots,x_m).
\end{align*}
Hence by Khimshiashvili's formula \cite{Khimshiashvili}, we have :
 $\varphi (q)=1-\deg_0 \nabla g'$,
where $\deg_ 0 \nabla g'$ is the topological degree of the map
 $\frac{\nabla g'}{\Vert \nabla g' \Vert} : S_\varepsilon^{m-n} \to
 S^{m-n}$.
Two cases are possible. If $\lambda$ is even then: 
$$
q \in A_k^+(f) \Leftrightarrow \frac{\partial^{k+1}g}{\partial
 x_n^{k+1}}(q) >0 \hbox{ and }  
q \in A_k^-(f) \Leftrightarrow \frac{\partial^{k+1}g}{\partial x_n^{k+1}}(q) <0.$$
If $\lambda$ is odd then: 
$$q \in A_k^+(f) \Leftrightarrow \frac{\partial^{k+1}g}{\partial x_n^{k+1}}(q) <0 \hbox{ and }  
q \in A_k^-(f) \Leftrightarrow \frac{\partial^{k+1}g}{\partial x_n^{k+1}}(q) >0.$$
Finally we see that the sets $A_k^+(f)$ and $A_k^-(f)$ are in
 correspondence with the sets 
$A_k(f) \cap 
\{\frac{\partial^{k+1}g}{\partial x_n^{k+1}}>0\}$ and $A_k(f) \cap 
\{\frac{\partial^{k+1}g}{\partial x_n^{k+1}}<0\}$, 
which enables us to conclude. 
\end{proof}

We can state our main theorem which is a slight improvement of a result
of T.~Fukuda \cite{Fukuda1988} for $N=\mathbb{R}^n$ and O.~Saeki
\cite{Saeki} for a general $N$. 

\begin{thm}\label{RThm6A}
Let $f : M^m \to N^n$ be a Morin mapping. Assume that $M$ is compact, $N$ is connected and $m-n$ is odd. Then we have:
$$
\chi(M)=\sum_{k: \odd}
\Bigl[\chi(\overline{A_k^+(f)})-\chi(\overline{A_k^-(f)})\Bigr]. 
$$
\end{thm}
\begin{proof}
Applying Corollary \ref{RMorin1}, we get:
$$
\chi_c(M)-\chi_f\, \chi_c(N)
=\medsum_{k:\odd}
\bigl[\chi_c( A_k^+ (f))-\chi_c( A_k^-(f))\bigr],
$$
where $\chi_f$ is the Euler characteristic of a regular fiber of $f$. In this situation, $\chi_f=0$ because the regular fiber of $f$ is a compact odd-dimensional manifold. Then we remark  that $\chi_c(M)=\chi(M)$ because $M$ is compact. Moreover by the additivity of the Euler-Poincar\'e with
compact support, we have:
\begin{align*}
\chi(\overline{A_k^+(f)})
=&\chi_c(\overline{A_k^+(f)})
=\chi_c(A_k^+(f))+\chi_c(\partial(\overline{A_k^+(f)}))
=\chi_c(A_k^+(f))+\chi_c(\overline{A_{k+1}(f)}),\\ 
\chi(\overline{A_k^-(f)})
=&\chi_c(\overline{A_k^-(f)})
=\chi_c(A_k^-(f))+\chi_c(\partial(\overline{A_k^-(f)}))
=\chi_c(A_k^-(f))+\chi_c(\overline{A_{k+1}(f)}).
\end{align*} 
This implies that 
$\chi(\overline{A_k^+(f)})-\chi(\overline{A_k^-(f)})
=\chi_c(A_k^+(f))-\chi_c(A_k^{-}(f))$. 
\end{proof}
We end this subsection with two remarks: 

\begin{enumerate}
\item If $m$ is odd then $n$ is even and $\chi(M)=0$. If $k$ is odd, the dimension of $\overline{A_k^+(f)}$ and
$\overline{A_k^+(f)}$ is odd. Furthermore, we have:
$$\chi(\overline{A_k^+(f)})=\frac{1}{2}\chi(\partial \overline{A_k^+(f)})=\frac{1}{2}\chi(\overline{A_{k+1}(f)})=
\frac{1}{2}\chi(\partial \overline{A_k^-(f)})=\chi(\overline{A_k^-(f)}),$$
and $\chi(\overline{A_k^+(f)})-\chi(\overline{A_k^-(f)})=0$. In this case, our theorem is trivial.
\item If $m$ is even and $n=1$, then we can apply our theorem. 
In this situation, there is only a finite number of singular
points, which are the elements of $A_1^+(f)$ and of $A_1^-(f)$. Theorem
      \ref{RThm6A} gives 
that $\chi(M)= \# A_1^+(f)-\# A_1^-(f)$. We recover the well-known Morse equalities.
\end{enumerate} 

\subsection{Morin  maps from $M^n$ to $N^n$}
Let $f : M^n \to N^n$ be a Morin mapping from a compact oriented manifold $M$ of dimension $n$ to a connected manifold $N$ of the same dimension.
For any $p \in M$, let $\varphi (p)$ be the local topological degree of
the map-germ $f : (M,p) \to (N,f(p))$. Recall that $\varphi(p)=0$ if $p
\in A_k(f)$ and $k$ odd and that $\vert \varphi(p) \vert=1$ if $p \in
A_k(f)$ and $k$ even. Hence, if $k$ is even, $A_k(f)$ splits into two
subsets $A_k^+(f)$ and $A_k^-(f)$ where $A_k^+(f)$ (resp. $A_k^-(f)$)
consists of the points $p$ such that $\varphi(p)=1$
(resp. $\varphi(p)=-1$). It is well known that the $A_k(f)$'s are smooth
manifolds of dimension $n-k$, that the $\overline{A_k(f)}$'s are smooth
manifolds with boundary and that:
$$
\overline{A_k(f)}= \cup_{i \ge k} A_i(f) , \  \
\partial{\overline{A_k(f)}}=\cup_{i>k} A_i(f).
$$
Remark that $A_0(f)$ is the set of regular points of $f$. 
Let  us  describe more precisely the structure of the sets $A_k^{\pm}(f)$. 
\begin{prop}
If $k$ is even, then $\overline{A_k^+(f)}$ and $\overline{A_k^-(f)}$ are manifolds with boundary of dimension $n-k$ and 
$\partial \overline{A_k^+(f)}=\partial \overline{A_k^-(f)}= \overline{A_{k+1}(f)}$.
\end{prop}
\begin{proof} 
Let $p$ be  a point in $A_k(f)$, $k$ even. In local coordinates, $f$ is given by:
$$
\begin{array}{l}
y_i \circ f = x_i \hbox{ for } i \le n-1 ,\cr
y_n \circ f = x_n^{k+1} + \sum_{i=1}^{k-1} x_ix_n^{k-i}. \cr 
\end{array}
$$
If we suppose that $(x_1,\ldots,x_n)$ and $(y_1,\ldots,y_n)$ are coordinates in direct basis,
then $f$ has two possible forms:
$$
\begin{cases}
y_i \circ f = &x_i \quad\hbox{ for } i \le n-1 ,\cr
y_n \circ f = &x_n^{k+1} + \sum_{i=1}^{k-1} x_ix_n^{k-i}, \cr 
\end{cases}
\quad\textrm{or}\quad
\begin{cases}
y_i \circ f = &x_i \quad\hbox{ for } i \le n-1 ,\cr
y_n \circ f = &-x_n^{k+1} + \sum_{i=1}^{k-1}(-1)^{k-i} x_ix_n^{k-i}. \cr 
\end{cases}
$$
In the first case, $\varphi(p)=1$ and in the second case $\varphi(p)=-1$. 

We can prove the fact that $A_k^+(f)$ and $A_k^-(f)$ are manifolds of dimension $n-k$  with the same method as
in Proposition \ref{RProp6A}. 
Now let $l>k$ and let $p \in A_l(f)$. Locally $f$ is given by:
$$\begin{array}{l}
y_i \circ f = x_i \hbox{ for } i \le n-1 ,\cr
y_n \circ f = \pm x_n^{l+1} + \sum_{i=1}^{l-1}\pm x_ix_n^{l-i}. \cr 
\end{array} $$
Let us denote by $g$ the function $y_n \circ f$. We have:
$$A_k(f)= \left\{ \frac{\partial g}{\partial x_n}=\cdots=\frac{\partial^k g}{\partial x_n^k}=0,\frac{\partial^{k+1} g}{\partial x_n^{k+1}} \not= 0
\right\},$$
and:
$$\overline{A_{k+1}(f)}= \left\{ \frac{\partial g}{\partial x_n}=\cdots=\frac{\partial^{k+1} g}{\partial x_n^{k+1}}=0
\right\}.$$
Let $q=(q_1,\ldots,q_n)$ be a point in $A_k(f)$ close to $p$. Let us find when $q \in A_k^+(f)$ or $q\in
A_k^-(f)$. For this we have to compute $\varphi(q)$. Let $\epsilon$ be a small real number and let us look for the
zeros lying close to $q$ of the following system:
$$\left\{ \begin{array}{l}
y_i \circ f (x)= q_i \hbox{ for } i \le n-1, \cr
g(x)=g(q)+ \epsilon .\cr
\end{array} \right.$$
This system is equivalent to:
$$\left\{ \begin{array}{l}
x_i= q_i \hbox{ for } i \le n-1 ,\cr
g(q_1,\ldots,q_{n-1},q_n+x_n')=g(q)+ \epsilon . \cr
\end{array} \right.$$
But: 
$$g(q_1,\ldots,q_{n-1},q_n+x_n')=g(q)+\sum_{i\ge k+1}\frac{\partial^i g}{\partial
x_n^i}(q){x_n'}^i.$$
Then we see  that $\varphi(q)= \hbox{sign} \frac{\partial^{k+1}g}{\partial x_n^{k+1}}(q)$. We conclude as in Proposition \ref{RProp6A}.
\end{proof}
\begin{thm}
Let $f : M^n \to N^n$ be a Morin mapping. Assume that $M$ is compact and
 oriented and that $N$ is connected and oriented. We have:
$$
\sum_{k : \e} \Bigl[
\chi(\overline{A_k^+(f)})-\chi(\overline{A_k^-(f)})\Bigr]
=(\deg f)\chi(N).
$$
\end{thm}
This is proved by I.~R.~Quine \cite{Quine} when $n=2$. 
It appeared in a preprint of I.~Nakai \cite{Nakai1996} for any $n$.
\begin{proof} 
By Corollary \ref{RMorin4}, we know that:
$$
\medsum_{k :\textrm{even}}
\bigl[\chi_c(A_k^+(f))-\chi_c(A_k^-(f))\bigr]
=(\deg f)\chi_c(N).
$$
If $N$ is compact then $\chi_c(N)=\chi(N)$ and if $N$ is not compact then $\deg f=0$. In both cases the equality $(\deg f) \chi_c (N)= (\deg f) \chi(N)$ is true. With the same arguments as in Theorem \ref{RThm6A}, it is easy to prove that $\chi(\overline{A_k^+(f)})-\chi(\overline{A_k^-(f)})=
\chi_c(A_k^+(f))-\chi_c(A_k^-(f))$.
\end{proof}

\begin{rem}
When $n$ is odd, $\overline{A_k^+(f)}$ and $\overline{A_k^-(f)}$
are odd-dimensional manifolds with the same boundary and so the left hand-side of the equality vanishes. But the right-hand
side is also zero because $\chi(N)=0$ if $N$ is compact and $\deg f=0$ if $N$ is not compact. Hence our theorem is trivial in
this case.
\end{rem}

\subsection{Local versions}
We give local versions of the global formulas of the previous subsections. 

We work first with map-germs $f :(\mathbb{R}^n,0) \rightarrow
(\mathbb{R}^p,0)$, $n>p$, which are generic in the sense of
Theorem $1'$ in \cite{Fukuda1985}. There are two cases: 

Case I) If the origin $0$ is not isolated in $f^{-1}(0)$, i.e $0 \in
\overline{f^{-1}(0) \setminus \{0\}}$, then
there exist a positive number $\varepsilon_0$ and a strictly increasing
function 
$\delta : [0,\varepsilon_0] \rightarrow[0,+\infty)$ 
with $\delta(0)=0$ such that for every $\varepsilon$ and
$\delta$ with $0< \varepsilon \le \varepsilon_0$ and
$0<\delta< \delta(\varepsilon)$ the following properties hold:
\begin{enumerate}
\item $f^{-1}(0) \cap S_\varepsilon^{n-1}$ is an $(n-p-1)$-dimensional
      manifold and it is diffeomorphic to $f^{-1}(0) \cap
S_{\varepsilon_0}^{n-1}$.
\item $\overline{B_\varepsilon^n} \cap f^{-1}(S_\delta^{p-1})$ is a smooth manifold with boundary and it is diffeomorphic to 
$\overline{B_{\varepsilon_0}^n} \cap f^{-1}(S_{\delta(\varepsilon_0)}^{p-1})$. 
\item $\partial(\overline{B_\varepsilon^n} \cap f^{-1}(\overline{B_\delta^{p}}))$ is homeomorphic to $S_\varepsilon^{n-1}$.
\item The restricted mapping $f : \overline{B_\varepsilon^n} \cap f^{-1}(S_\delta^{p-1}) \rightarrow S_\delta^{p-1}$ is topologically
stable ($C^\infty$ stable if $(n,p)$ is a nice pair) and its topological type is independent of $\varepsilon$ and $\delta$.
\end{enumerate}
Here $B_\varepsilon^n$ denotes the open ball of radius $\varepsilon$
centered at $0$ and $S_\varepsilon^{n-1}$ the sphere of radius
$\varepsilon$ centered at $0$ in $\mathbb{R}^n$.

Case II) If the origin $0$ is isolated in $f^{-1}(0)$, i.e $0 \notin \overline{f^{-1}(0) \setminus \{0\}}$, then
there exists a positive number $\varepsilon_0$ such that for every $\varepsilon$ with $0< \varepsilon \le \varepsilon_0$
the following properties hold:
\begin{enumerate}
\item $f^{-1}(S_\varepsilon^{p-1})$ is diffeomorphic to 
$S_\varepsilon^{n-1}$.
\item The restricted mapping $f : f^{-1}(S_\varepsilon^{p-1}) \rightarrow S_\varepsilon^{p-1}$ is topologically
stable ($C^\infty$ stable if $(n,p)$ is a nice pair) and its topological type is independent of $\varepsilon$.
\end{enumerate}

We will focus first on Case I). Note that in this case, $\overline{B_\varepsilon^n} \cap f^{-1}(\overline{B_\delta^{p}})$ is a manifold with corners
whose topological boundary is the manifold with corners $\overline{B_\varepsilon^n} \cap f^{-1}(S_\delta^{p-1}) \cup S^{n-1}_\varepsilon \cap f^{-1}(\overline{B_\delta^p})$. 
We will use the following notations : $B_{\varepsilon,\delta}=f^{-1}(\overline{B_\delta^p}) \cap \overline{B_\varepsilon^n}$,
$\partial B_{\varepsilon,\delta}= \overline{B_\varepsilon^n} \cap f^{-1}(S_\delta^{p-1}) \cup S^{n-1}_\varepsilon \cap f^{-1}(\overline{B_\delta^p})$,
$C_{\varepsilon,\delta}= B_\varepsilon^n \cap f^{-1}(S_\delta^{p-1})$ and 
$I_{\varepsilon,\delta}$ is the topological interior of $B_{\varepsilon,\delta}$.

Let us denote by $\partial f$ the restricted mapping $f_{\vert C_{\varepsilon,\delta} } : C_{\varepsilon,\delta} \to S_\delta^{p-1}$ and let us assume that it is a Morin mapping. 
Let us consider a perturbation $\tilde{f}$ of $f$ such that  $\tilde{f}_{\vert I_{\varepsilon,\delta}}  :
I_{\varepsilon,\delta} \to B_\delta^{p}$ is a Morin mapping and $\tilde{f}=f$ in a neighborhood of $C_{\varepsilon,\delta}$.

Our aim is to generalize Theorem 2 of \cite{Fukuda1985} which deals with
map-germs from $\mathbb{R}^n$ to $\mathbb{R}$, i.e to relate
the topology of $\Lk(f)=f^{-1}(0) \cap S_\varepsilon^{n-1}$ to the
topology of the singular set of $\tilde{f}$ and to the topology of the
singular set of $\partial f$. As in the previous sections,  
we will denote by $A_k(\tilde{f})$ (resp. $A_k(\partial f)$), the set of
singular points of $\tilde{f}$ (resp. $\tilde{f}$) of type $A_k$. 
The first result is a local version of Saeki's formula 
(Theorem 2.3 in \cite{Saeki}). 
 
\begin{thm}\label{RThm6C}
We have:
$$
\psi(\Lk(f))
\equiv
1+\sum_{k=1}^{p-1}
\psi(A_k(\partial f))+\# A_p(\tilde{f}) \bmod 2,
$$
where $\psi$ denotes the semi-characteristic.
\end{thm} 
\begin{proof} 
Note that for $\tilde{\delta}$ a sufficiently small regular value of
$\tilde{f}$ ($\vert \tilde{\delta} \vert \le \delta$), we have: 
$$
\chi_c(\tilde{f}^{-1}(\tilde{\delta})\cap I_{\varepsilon,\delta}) 
\equiv
\chi(\tilde{f}^{-1}(\tilde{\delta})\cap I_{\varepsilon,\delta})
\equiv
\chi(\tilde{f}^{-1}(\tilde{\delta}) \cap B_{\varepsilon,\delta})
\equiv
\psi(\tilde{f}^{-1}(\tilde{\delta}) \cap S_\varepsilon^{n-1})
\equiv\psi(\Lk(f)) \mod{2}.
$$
The last equality comes from the fact that $f$ has an isolated
 singularity, that $\tilde{f}^{-1}(\tilde{\delta})$
intersects $S_\varepsilon^{n-1}$ transversally and that $\tilde{f}$ is
 close to $f$.

On the one hand, applying Theorem \ref{RThm1},  Theorem \ref{RThm1B} and their corollaries to the restriction of $\tilde{f}$ to $I_{\varepsilon,\delta}$, we obtain:
$$
\medsum_{k:\e} \chi_c(A_k(\tilde{f} )\cap
I_{\varepsilon,\delta}) \equiv \psi(\Lk(f)) \bmod 2.
$$

On the other hand, by additivity, we have:
$$
1 
\equiv\chi_c (I_{\varepsilon,\delta}) 
\equiv\sum_k \chi_c(I_{\varepsilon,\delta} \cap A_k(\tilde{f})) \bmod 2.
$$
For each $k \ge 1$, we have:
$$\overline{A_k(\tilde{f}) \cap I_{\varepsilon,\delta}} =
A_k(\tilde{f})\cap I_{\varepsilon,\delta} \sqcup \overline{A_{k+1}(\tilde{f})} \cap I_{\varepsilon,\delta} \sqcup 
\overline{A_k(\tilde{f})} \cap C_{\varepsilon,\delta},$$ 
because if $\varepsilon$ and $\delta$ are small enough the singular set of $\tilde{f}$ does not intersect $f^{-1}(\overline{B_\delta^p}) \cap S_\varepsilon^{n-1}$.
Before carrying on with our computations, let us observe that for 
$k \in \{1,\ldots,p-1\}$, $\overline{A_k(\tilde{f})} \cap
 C_{\varepsilon,\delta} = A_k(\partial f)$. 
It is not difficult to see 
 this with the characterization of the $A_k$ sets by the ranks of the
 iterate jacobians. Moreover, using the characterization of the $A_k^+$
 and $A_k^-$ sets by the Euler characteristic of the nearby fiber, we
 can say that 
$\overline{A_k^+(\tilde{f})} \cap C_{\varepsilon,\delta}= A_k^+(\partial
 f)$ 
and 
$\overline{A_k^-(\tilde{f})} \cap C_{\varepsilon,\delta} = A_k^-(\partial f)$.
Hence:
\begin{align*}
\chi(\overline{A_k(\tilde{f})\cap I_{\varepsilon,\delta}})
\equiv&\chi_c(\overline{A_k(\tilde{f}) \cap I_{\varepsilon,\delta}})
\\
\equiv&\chi_c(A_k(\tilde{f}) \cap I_{\varepsilon,\delta}) 
+\chi_c(\overline{A_{k+1}(\tilde{f})} \cap I_{\varepsilon,\delta})
+\chi_c(\overline{A_k(\tilde{f})}\cap C_{\varepsilon,\delta})
\\
\equiv&
\chi_c(A_k(\tilde{f}) \cap I_{\varepsilon,\delta})
+\chi_c(\overline{A_{k+1}(\tilde{f})} \cap I_{\varepsilon,\delta}) 
\pmod 2,
\end{align*}
because $\overline{A_k(\tilde{f})} \cap C_{\varepsilon,\delta}$ is a compact boundary. 
Furthermore, we have:
$$\chi(\overline{A_{k+1}(\tilde{f}) \cap I_{\varepsilon,\delta}})=\chi_c (\overline{A_{k+1}(\tilde{f})} \cap 
I_{\varepsilon,\delta} )+\chi_c (\overline{A_{k+1}(\tilde{f})} \cap 
C_{\varepsilon,\delta})=\chi_c(\overline{A_{k+1}(\tilde{f})} \cap 
I_{\varepsilon,\delta} ).$$
Finally, for each $k$, 
$
\chi_c(A_k(\tilde{f}) \cap I_{\varepsilon,\delta})
=
\chi(\overline{A_k(\tilde{f}) \cap I_{\varepsilon,\delta}})
+
\chi(\overline{A_{k+1}(\tilde{f}) \cap I_{\varepsilon,\delta}}),
$
and so:
\begin{align*}
\psi(\Lk(f))\equiv&
1+\sum_{k=1}^{p} \chi(\overline{A_k(\tilde{f}) \cap
 I_{\varepsilon,\delta}}) \bmod 2,\\
\psi(\Lk(f))\equiv& 
1+\sum_{k=1}^{p-1} \psi (A_k(\partial f) )+\# A_p(\tilde{f})
\bmod 2.
\end{align*}
\end{proof} 
Let us examine some special cases. 
When $p=1$, we find:
$$
\psi(\Lk(f))\equiv 1+\#A_1(\tilde{f})\equiv1+\deg_0 \nabla f \pmod{2},
$$
where $\deg_ 0 \nabla f$ is the topological degree of the map 
$\frac{\nabla f}{\Vert \nabla f \Vert} : S_\varepsilon^{n-1} \to
S^{n-1}$. This due to the fact 
$\tilde{f}$ is a Morse function and the points in 
$A_1(\tilde{f})$ are exactly its critical points. 
When $p=2$, we find: 
$$
\psi(\Lk(f))\equiv1 +\psi(A_1(\partial f))+\#A_2(\tilde{f}) \bmod 2.
$$
If $\tilde{f}$ is close to $f$ then $\psi(A_1(\partial f))$ is equal to 
$\frac{1}{2}b(C(f)) \bmod 2$ where $C(f)$ denotes critical locus of $f$ and $b(C(f))$ the number of branches of $C(f)$. 
Hence:
$$
\psi(\Lk(f))\equiv1+\frac{1}{2}b(C(f)) +\# A_2(\tilde{f}) \bmod 2.
$$
Since $b(C(f))$ is a topological invariant of $f$, we deduce that $\# A_2 (\tilde{f}) \bmod 2$ is a topological invariant of $f$.
Similarly if $p=3$, this gives:
$$
\psi(\Lk(f))\equiv1+\psi(C(f) \cap \partial B_{\varepsilon,\delta})
+\frac{1}{2}b(C(f_{\vert C(f)})) +\# A_3(\tilde{f}) \bmod 2.
$$
In the sequel, we will improve Theorem \ref{RThm6C} in some situations. 
Let us assume that $n-p$ is odd.
\begin{thm}\label{RThm6D}
If $n-p$ is odd, then we have:
$$
\chi(\Lk(f))
=2-2 \sum_{k: \odd}
\Bigl[
\chi(\overline{A_k^+(\tilde{f}) \cap I_{\varepsilon,\delta}})
-\chi(\overline{A_k^-(\tilde{f}) \cap I_{\varepsilon,\delta}})
\Bigr].
$$
Furthermore, when $n$ is odd and $p$ is even, we have:
$$
\chi(\Lk(f))=
2-\sum_{k: \odd} 
\Bigl[
\chi(\overline{A_k^+(\partial f)}) 
-\chi(\overline{A_k^-(\partial f)})
\Bigr].
$$
\end{thm}
\begin{proof}
With the same notations as in Theorem \ref{RThm6C}, we can write:
$$
\chi_c(\tilde{f}^{-1}(\tilde{\delta})\cap B_{\varepsilon,\delta}) 
= \chi_c(\tilde{f}^{-1}(\tilde{\delta})\cap I_{\varepsilon,\delta} )
+\chi_c(\tilde{f}^{-1}(\tilde{\delta})\cap\partial B_{\varepsilon,\delta}),
$$ 
thus:
$$
\frac{1}{2} \chi(\Lk(f))
=\chi_c(\tilde{f}^{-1}(\tilde{\delta}) 
\cap I_{\varepsilon,\delta})+\chi(\Lk(f)).
$$
Therefore, we get:
$$
\chi_c(\tilde{f}^{-1}(\tilde{\delta}) \cap
 I_{\varepsilon,\delta})=-\frac{1}{2} \chi(\Lk(f)).
$$
Applying Corollary \ref{RMorin1}, we obtain:
$$
\chi_c(I_{\varepsilon,\delta})
+\frac{1}{2} \chi(\Lk(f)) \chi_c(B_\delta^p)
=\sum_{k : \textrm{odd}}\chi_c (A_k^+(\tilde{f})\cap 
I_{\varepsilon,\delta})-\chi_c(A_k^-(\tilde{f}) \cap I_{\varepsilon,\delta}).
$$ 
Let us compute $\chi_c(I_{\varepsilon,\delta})$. 
We have:
$$
 \chi(B_{\varepsilon,\delta})=\chi_c(B_{\varepsilon,\delta})=
\chi_c(I_{\varepsilon,\delta} ) + 
\chi_c(B^n_\varepsilon \cap f^{-1}(S^{p-1}_\delta)) +
\chi_c(S^{n-1}_\varepsilon \cap f^{-1} (\overline{B^p_\delta})). $$
If $n$ is odd and $p$ is even, we have:
$$\chi(B_{\varepsilon,\delta})=\frac{1}{2}\chi(S^{n-1}_\varepsilon \cap f^{-1}(\overline{B^p_\delta}))+
\frac{1}{2}\chi(\overline{B^n_\varepsilon}
\cap f^{-1}(S^{p-1}_\delta)),$$
and:
\begin{align*}
\chi(\overline{B^n_\varepsilon }\cap f^{-1}(S^{p-1}_\delta))
=&\chi_c(\overline{B^n_\varepsilon} \cap f^{-1}(S^{p-1}_\delta))\\
=&
\chi_c(B^n_\varepsilon \cap f^{-1}(S^{p-1}_\delta)) 
+ \chi_c(S^{n-1}_\varepsilon \cap f^{-1}(S^{p-1}_\delta))
\\
=&
\chi_c(B^n_\varepsilon \cap f^{-1}(S^{p-1}_\delta))
+\chi(S^{n-1}_\varepsilon \cap f^{-1}(S^{p-1}_\delta)) 
\\
=&
\chi_c (B^n_\varepsilon  \cap f^{-1}(S^{p-1}_\delta)).
\end{align*}
Thus we obtain:
\begin{align*}
\chi_c(I_{\varepsilon,\delta})
=&
\textrm{\small$
\frac{1}{2}\chi(S^{n-1}_\varepsilon \cap f^{-1}( \overline{B^p_\delta)})
+\frac{1}{2}\chi(\overline{B^n_\varepsilon}\cap f^{-1} (S^{p-1}_\delta))
-\chi (\overline{B^n_\varepsilon} \cap f^{-1}(S^{p-1}_\delta))
-\chi(S^{n-1}_\varepsilon \cap f^{-1}(\overline{B^p_\delta}))$} \\
=&
-\frac{1}{2} \chi(\overline{B^n_\varepsilon} \cap f^{-1}(S^{p-1}_\delta)) -\frac{1}{2} \chi(S^{n-1}_\varepsilon \cap
f^{-1}(\overline{B^p_\delta})) 
=-\chi(B_{\varepsilon,\delta})=-1.
\end{align*} 
Finally we get:
\begin{align*}
\frac{1}{2}\chi(\Lk(f))
=&1+ 
\sum_{k: \odd}
\Bigl[
\chi_c(A_k^+(\tilde{f})\cap I_{\varepsilon,\delta}) 
-\chi_c(A_k^-(\tilde{f})\cap I_{\varepsilon,\delta})
\Bigr],
\end{align*}
which means:
$$
\chi(\Lk(f))
= 2 +2
\sum_{k : \odd} 
\Bigl[
 \chi_c(A_k^+(\tilde{f})\cap I_{\varepsilon,\delta}) 
-\chi_c(A_k^-(\tilde{f})\cap I_{\varepsilon,\delta})
\Bigr].
$$
Since dim $A_k^+(\tilde{f})=$ dim $A_k^-(\tilde{f})=p-k$ is odd, 
we can establish using the same arguments as above that:
\begin{align*}
\chi_c(A_k^+(\tilde{f}) \cap I_{\varepsilon,\delta})
=&-\chi(\overline{A_k^+(\tilde{f}) \cap I_{\varepsilon,\delta}})
=
-\frac{1}{2}\chi(\overline{A_k^+(\tilde{f})}\cap C_{\varepsilon,\delta})
-\frac{1}{2}\chi(\overline{A_{k+1}(\tilde{f})\cap I_{\varepsilon,\delta}}),\\
\chi_c(A_k^-(\tilde{f})\cap I_{\varepsilon,\delta})
=&-\chi(\overline{A_k^-(\tilde{f}) \cap I_{\varepsilon,\delta}})
=-\frac{1}{2}\chi(\overline{A_k^-(\tilde{f})}\cap
 C_{\varepsilon,\delta}) -\frac{1}{2}
\chi(\overline{A_{k+1}(\tilde{f}) \cap I_{\varepsilon,\delta} }).
\end{align*}
Finally, we obtain:
$$
\chi(\Lk(f))
= 
2-2\sum_{k : \odd}
\Bigl[
 \chi(\overline{A_k^+(\tilde{f})\cap I_{\varepsilon,\delta}}) 
-\chi(\overline{A_k^-(\tilde{f})\cap I_{\varepsilon,\delta}})
\Bigr]
=
2-\sum_{k : \odd} 
\Bigl[
\chi(\overline{A_k^+(\partial f)}) 
-\chi(\overline{A_k^-(\partial f)})
\Bigr].
$$
If $n$ is even and $p$ is odd, then:
\begin{align*}
\chi_c(B_\varepsilon^n \cap f^{-1}(S^{p-1}_\delta))
=&-\chi(B_{\varepsilon}^n \cap f^{-1}(S^{p-1}_\delta))
=-\frac{1}{2} \chi (S^{n-1}_\varepsilon\cap f^{-1}(S^{p-1}_\delta)),
\\
\chi_c (S^{n-1}_\varepsilon \cap f^{-1}(\overline{B^p_\delta}))
=&\chi(S^{n-1}_\varepsilon \cap f^{-1}(\overline{B^p_\delta}))
=\frac{1}{2}\chi(S^{n-1}_\varepsilon \cap f^{-1}(S^{p-1}_\delta)).
\end{align*}
So
$\chi_c(I_{\varepsilon,\delta})=\chi(B_{\varepsilon,\delta})=1$,  
and:
$$
1-\sum_{k: \odd} 
\chi_c(A_k^+(\tilde{f})\cap I_{\varepsilon,\delta}) 
+\sum_{k: \odd} 
\chi_c(A_k^-(\tilde{f})\cap I_{\varepsilon,\delta}) 
=
\frac{1}{2} \chi (\Lk(f)),
$$ 
and then:
$$
\chi(\Lk(f))
=
2-2 \left(
\sum_{k: \odd} 
\chi_c(A_k^+(\tilde{f})\cap I_{\varepsilon,\delta}) 
-\sum_{k: \odd} 
\chi_c(A_k^-(\tilde{f})\cap I_{\varepsilon,\delta})
\right).
$$ 
Here $\dim A_k^+(\tilde{f})=\dim A_k^-(\tilde{f})=p-k$ is even 
when $k$ is odd. We have:
\begin{align*}
\chi(\overline{A_k^+(\tilde{f}) \cap I_{\varepsilon,\delta}}) 
=&
\chi_c(\overline{A_k^+(\tilde{f})\cap I_{\varepsilon,\delta}}) \\
=&
\chi_c(A_k^+(\tilde{f})\cap I_{\varepsilon,\delta}) 
+
\chi_c(\overline{A_{k+1}(\tilde{f})}\cap I_{\varepsilon,\delta}) 
+
\chi_c(\overline{A_k^+(\tilde{f})}\cap C_{\varepsilon,\delta}) \\
=&
\chi_c(A_k^+(\tilde{f}) \cap I_{\varepsilon,\delta}) 
+
\chi_c(\overline{A_{k+1}(\tilde{f})} \cap I_{\varepsilon,\delta}) 
+
\chi(\overline{A_k^+(\tilde{f})} \cap C_{\varepsilon,\delta})\\
=&
\chi_c(A_k^+(\tilde{f}) \cap I_{\varepsilon,\delta}) 
+
\chi_c(\overline{A_{k+1}(\tilde{f})} \cap I_{\varepsilon,\delta}).
\end{align*}
Hence:
$$
\chi(\overline{A_k^+(\tilde{f}) \cap I_{\varepsilon,\delta}})
-\chi(\overline{A_k^-(\tilde{f}) \cap I_{\varepsilon,\delta}})
=\chi_c(A_k^+(\tilde{f}) \cap I_{\varepsilon,\delta})
-\chi(A_k^-(\tilde{f}) \cap I_{\varepsilon,\delta}).
$$
\end{proof}

The same results hold in Case II) replacing $B^n_\varepsilon \cap f^{-1}(S_\varepsilon^{p-1})$ with
$f^{-1}(S_\varepsilon^{p-1})$, which is diffeomorphic to
$S_\varepsilon^{n-1}$,  $B_{\varepsilon,\delta}$ with $f^{-1}(\overline{B^p_\varepsilon})$, $I_{\varepsilon,\delta}$ with the topological interior of $f^{-1}(\overline{B^p_\varepsilon})$ and $\chi(\Lk(f)))$ with $0$.

Now we work with map-germs from $(\mathbb{R}^n,0)$ to $(\mathbb{R}^n,0)$. Let $f : (\mathbb{R}^n,0) \rightarrow
(\mathbb{R}^n,0)$ be a map-germ such that $0$ is isolated in $f^{-1}(0)$. We assume that $f$ is generic in the sense of
Theorem 3 in \cite{Fukuda1981} : there exists a positive number $\varepsilon_0$ such that for any number $\varepsilon$ with
$0<\varepsilon \le \varepsilon_0$, we have:
\begin{enumerate}
\item $\tilde{S}^{n-1}_\varepsilon = f^{-1}(S_\varepsilon^{n-1})$ is a homotopy $(n-1)$-sphere which, if $n \not= 4,5$, is
diffeomorphic to the natural $(n-1)$-sphere $S^{n-1}$,
\item the restricted mapping $f_{\vert \tilde{S}_\varepsilon^{n-1}} : \tilde{S}_\varepsilon^{n-1} \rightarrow
S_\varepsilon^{n-1}$ is topological stable ($C^\infty$ stable if $(n,p)$ is a nice pair),
\item letting $\tilde{B}_\varepsilon^n=f^{-1}(\overline{B^n_\varepsilon})$, the restricted mapping $f_{\vert \tilde{B}^n_\varepsilon} : 
\tilde{B}_\varepsilon^n \setminus \{0\} \rightarrow \overline{B_\varepsilon^n} \setminus \{0\}$ is proper, topologically stable
($C^\infty$ stable if $(n,p)$ is nice) and topologically equivalent ($C^\infty$ equivalent if $(n,p)$ is nice) to the product
mapping:
$$
(f_{\vert \tilde{S}^{n-1}_\varepsilon}) \times
      \Id_{(0,\varepsilon)} : \tilde{S}_\varepsilon^{n-1} \times
      (0,\varepsilon)
\rightarrow S_\varepsilon^{n-1} \times (0,\varepsilon),
$$
defined by $(x,t) \mapsto (f(x),t)$, 
\item consequently, $f_{\vert \tilde{B}^n_\varepsilon} : \tilde{B}^n_\varepsilon \rightarrow \overline{B^n_\varepsilon}$ is topologically equivalent
to the cone:
$$C(f_{\vert \tilde{S}_\varepsilon^{n-1}}) :  \tilde{S}_\varepsilon^{n-1} \times [0,\varepsilon) /
\tilde{S}_\varepsilon^{n-1} \times \{0\} \rightarrow  S_\varepsilon^{n-1} \times [0,\varepsilon) /
S_\varepsilon^{n-1} \times \{0\},$$
of the stable mapping $f_{\vert \tilde{S}_\varepsilon^{n-1}} : \tilde{S}_\varepsilon^{n-1} \rightarrow S_\varepsilon^{n-1}$
defined by $$C(f_{\vert \tilde{S}_\varepsilon^{n-1}}) (x,t)=(f(x),t).$$ 
\end{enumerate}
Note that in this case $\tilde{B}_\varepsilon=f^{-1}(\overline{B_\varepsilon^n})$ is a smooth manifold with boundary
$f^{-1}(S_\varepsilon^{n-1})$. This last manifold has the homotopy type of $S^{n-1}$. 

We will keep the notations of the previous sections.
We denote by $\tilde{B}_\varepsilon$ the set $f^{-1}(\overline{B^n_\varepsilon})$, by $\tilde{I}_\varepsilon$ its
topological interior and by $\partial \tilde{B}_\varepsilon$ its
boundary. We denote by $\partial f$ the restricted mapping $f _{\vert
\partial \tilde{B}_\varepsilon} : \partial \tilde{B}_\varepsilon \to
S^{n-1}_\varepsilon$ and we assume that  it is a Morin mapping.

Let us consider a perturbation $\tilde{f}$ of $f$ such that
$\tilde{f}_{\vert \tilde{I}_\varepsilon}:
\tilde{I}_\varepsilon \rightarrow  B_\varepsilon^{n}$ is a Morin mapping
and $\tilde{f}=f$ in a neighborhood of $\partial
\tilde{B}_\varepsilon$. 

The main
result is a local version of Corollary \ref{RMorin4}. 
\begin{thm}\label{RThm6E} 
We have:
$$
\deg_0 f =\sum_{k:\e}
\Bigl[
\chi(\overline{A_k^+(\tilde{f}) \cap \tilde{I}_\varepsilon})
-\chi(\overline{A_k^-(\tilde{f}) \cap \tilde{I}_\varepsilon})
\Bigr],
$$
where $\deg_0 f$ is the local topological degree of $f$ at $0$.
\end{thm}
\begin{proof}
Using Corollary \ref{RMorin4}, we obtain:
$$
(\deg_0 f)(-1)^n = 
\sum_{k : \textrm{even}} \chi_c (A_k^+(\tilde{f}) \cap \tilde{I}_\varepsilon) -
\chi_c (A_k^-(\tilde{f}) \cap \tilde{I}_\varepsilon).
$$
It remains to relate the Euler characteristics with compact support to the topological Euler characteristics. But, as in Theorem
\ref{RThm6D}, we have:
$$
\chi_c (A_k^+(\tilde{f}) \cap \tilde{I}_\varepsilon) -
\chi_c (A_k^-(\tilde{f}) \cap \tilde{I}_\varepsilon)=  (-1)^{n-k} \left(
\chi(\overline{A_k^+(\tilde{f}) \cap \tilde{I}_\varepsilon})-
\chi(\overline{A_k^-(\tilde{f}) \cap \tilde{I}_\varepsilon}) \right) .$$
\end{proof}

\begin{cor}
If $n$ is odd, we have: 
$$ 
2\deg_0 f =
\sum_{k : \textrm{even}} 
\Bigl[
 \chi(\overline{A_k^+(\partial f)})
-\chi(\overline{A_k^-(\partial f)})
\Bigr].
$$
\end{cor}

\begin{cor}
$$
\deg_0 f\equiv1+\sum_k \psi(\overline{A_k(\partial f)})+\# A_n(\tilde{f})
 \bmod 2.
$$
\end{cor}
\begin{proof}
 We have:
$$
1
=\chi(\tilde{B}_\varepsilon)
=\chi(\overline{A_0^+(\tilde{f})\cap \tilde{I}_\varepsilon})
+\chi(\overline{A_0^-(\tilde{f})\cap \tilde{I}_\varepsilon})
-\chi(\overline{A_1(\tilde{f}) \cap \tilde{I}_\varepsilon}),
$$
hence:
$$
\chi(\overline{A_0^+(\tilde{f})\cap\tilde{I}_\varepsilon})
-\chi(\overline{A_0^-(\tilde{f})\cap\tilde{I}_\varepsilon}) 
\equiv
1+\chi(\overline{A_1(\tilde{f})\cap\tilde{I}_\varepsilon})
\equiv
1 +\psi(\overline{A_1(\partial f)}) \bmod 2.
$$ 
Similarly, if $k$ is even and dim $A_k >0$, then:
$$
\chi(\overline{A_k^+(\tilde{f})\cap \tilde{I}_\varepsilon}) 
-\chi(\overline{A_k^-(\tilde{f})\cap \tilde{I}_\varepsilon}) 
\equiv
\chi(\overline{A_k(\tilde{f})\cap \tilde{I}_\varepsilon}) 
+\chi(\overline{A_{k+1}(\tilde{f})\cap \tilde{I}_\varepsilon}) 
\bmod 2.
$$ 
Thus we obtain that:  
$$
\chi(\overline{A_k^+(\tilde{f})\cap \tilde{I}_\varepsilon})
-\chi(\overline{A_k^-(\tilde{f})\cap \tilde{I}_\varepsilon})
\equiv
\begin{cases}
\psi(\overline{A_k(\partial f)})
+\psi(\overline{A_{k+1}(\partial f)}) &
\textrm{if dim $A_k(\tilde{f})>1$}\\
\psi(\overline{A_k(\partial f)})+\# A_{k+1}(\tilde{f}) 
&
\textrm{if dim $A_k(\tilde{f})=1$}\\
\# A_k(\tilde{f})
&\textrm{if dim $A_k(\tilde{f})=0$}
\end{cases}
$$
modulo 2. 
\end{proof} 

If $n=2$, this gives:
$$
\deg_0 f \equiv 1+\frac{1}{2} b(C(f)) + \# A_2(\tilde{f}) \bmod 2,
$$
and we recover Theorem 2.1 of T.~Fukuda and 
G.~Ishikawa \cite{FukudaIshikawa}.

If $n=3$, this gives:
$$
\deg_0 f\equiv1+\psi (C(f) \cap \partial \tilde{B}_\varepsilon) +
\frac{1}{2} b(C(f_{\vert C(f)})) + \# A_3(\tilde{f}) \bmod 2.
$$

\lsection{Complex maps}\label{ComplexMap}
We end with some remarks in the complex case.
Let $f:M\to N$ be a holomorphic map between complex manifolds $M$ and $N$. 
We assume that $N$ is connected. 
We assume that $f$ is locally infinitesimally stable in J.~Mather's 
sense. 

Let $c_\sigma$ denote the Euler characteristic of the local generic 
fiber of the map-germ of singular type $\sigma$. 
Let $\chi_f$ denote the Euler characteristics of the generic 
fibers of $f$.
\begin{thm}
If a locally infinitesimally stable map $f:M\to N$ 
does not have singularities at infinity, 
then 
$$
\medsum_{\sigma}c_\sigma\,\chi_c(M_\sigma(f))
=\chi_f\, \chi_c(N).
$$
\end{thm}
\begin{proof}
Apply Corollary \ref{Fubini3}. 
\end{proof}
\begin{cor}\label{CMorin}
If a Morin map $f:M\to N$ is locally trivial at infinity, then: 
$$
\chi_c(M)+(-1)^{m-n}\medsum_{k=1}^{n}\chi_c(\overline{A_k(f)})
=\chi_f\, \chi_c(N)
$$
where $m$ denotes the complex dimension of $M$ 
and $n$ denotes the complex dimension of $N$. 
\end{cor}
We should remark that this formula was firstly formulated by Y.~Yomdin 
(see \cite{Yomdin83}). 
Note also that when $m=n$, then $\chi_f$ is also the topological degree of $f$.

Let $f=(f_1,f_2):(\BB{C}^2,0)\to(\BB{C}^2,0)$ be a holomorphic map-germ
with $c(f)<\infty$ where:
$$
c(f)=\dim_{\BB{C}}\mathcal O_2\bigg/I_2
\begin{pmatrix}
\pd{f_1}{x_1}&\pd{f_2}{x_1}&\pd{J}{x_1}\\
\pd{f_1}{x_2}&\pd{f_2}{x_2}&\pd{J}{x_2}
\end{pmatrix}, 
\qquad
J=
\begin{vmatrix}
\pd{f_1}{x_1}&\pd{f_2}{x_1}\\
\pd{f_1}{x_2}&\pd{f_2}{x_2}
\end{vmatrix}.
$$
\begin{cor}[{\cite[(1.8)]{GaffneyMond}}]
Let $f,g:(\BB{C}^2,0)\to(\BB{C}^2,0)$ be holomorphic map-germs with 
 $c(f)<\infty$, $c(g)<\infty$. 
Let $f_t$, $g_t$ denote stable perturbations of $f$, $g$. 
If $f$ and $g$ are topologically right-left equivalent, then 
$\#A_2(f_t)=\#A_2(g_t)$. 
\end{cor}
\begin{proof}
Since the critical set 
can be characterized topologically, 
$(\BB{C}^2,\Sigma(f),0)$ and $(\BB{C}^2,\Sigma(g),0)$ are 
topologically equivalent, 
and they have the same Milnor number. 
Thus their smoothings have the same Euler characteristic and  
$\chi(\overline{A_1(f_t)})=\chi(\overline{A_1(g_t)})$.  
By Corollary \ref{CMorin}, we have: 
\begin{align*}
1+\chi_c(\overline{A_1(f_t)})+\#A_2(f_t)=&\deg_0 f,\\
1+\chi_c(\overline{A_1(g_t)})+\#A_2(g_t)=&\deg_0 g,
\end{align*}
and, since $\deg_0 f=\deg_0 g$, we conclude the result.
\end{proof}

\begin{rem}
Consider the map germ $f=(f_1,f_2):(\BB{C}^n,0)\to(\BB{C}^2,0)$, $n>2$.
Take a stable perturbation $f_t$ of $f$.
We have:
\begin{equation}\label{eqCM}
\chi_c(\overline{A_1(f_t)})+\#A_2(f_t)=(-1)^n(\chi_f-1).
\end{equation}
Consider the map 
$F:(\BB{C}^n,0)\times(\BB{C},0)\to(\BB{C}^2,0)\times(\BB{C},0)$ 
defined by $F(x,t)=(f_t(x),t)$. 
Since $\overline{A_1(F)}$ is determinantal, 
it is Cohen-Macaulay.  
So the map 
$\overline{A_1(F)}\to(\BB{C},0)$, $(x,t)\mapsto t$, is flat.  
So $\overline{A_1(f_t)}$ is a smoothing of $\overline{A_1(f)}$ and 
its Euler characteristic is described by the Milnor number of 
$\overline{A_1(f)}$: 
$\chi_c(\overline{A_1(f_t)})=1-\mu(\overline{A_1(f)})$, 
and we conclude that 
$\mu(\overline{A_1(f)})$ and $\chi_f$ determine $\#(A_2(f_t))$. 
Now we assume that $f$ is $\mathcal A$-finite. 
Then, we have: 
\begin{align*}
1-\mu(\overline{A_1(f)})=\chi_c(\overline{A_1(f_t)})
=&\chi_c(f_t(\overline{A_1(f_t)}))+d(f_t)\\
=&1-\mu(f(\overline{A_1(f)}))+2\#(A_2(f_t))+2d(f_t),
\end{align*}
where $d(f_t)$ denotes the number of double fold ($A_{1,1}$) 
points of $f_t$ nearby $0$. 
Combining this with \eqref{eqCM}, we obtain: 
$$
3\#A_2(f_t)+2d(f_t)=\mu(f(\overline{A_1}(f)))-1+(-1)^n(\chi_f-1).
$$
We conclude that 
$3\#A_2(f_t)+2d(f_t)$ (and thus $\#A_2(f_t)\bmod 2$) 
is a topological invariant of $f$. 
\end{rem}
\begin{rem}
Consider a map germ 
$f:(\BB{C}^3,0)\to(\BB{C}^3,0)$, $x\mapsto y=f(x)$. 
Take a stable perturbation $f_t$ of $f$. Then we obtain: 
$$
1+\chi_c(\overline{A_1(f_t)})+\chi_c (\overline{A_2(f_t)})+\#A_3(f_t)
=\deg f.
$$
Consider the map  
$F:(\BB{C}^3,0)\times(\BB{C},0)\to(\BB{C}^3,0)\times(\BB{C},0)$ 
defined by $F(x,t)=(f_t(x),t)$. 
Since $\overline{A_2(F)}$ is determinantal, it is Cohen-Macaulay. 
We obtain that the map 
$\overline{A_2(F)}\to(\BB{C},0)$, $(x,t)\mapsto t$, is flat.  
So $\overline{A_2(f_t)}$ is a smoothing of $\overline{A_2(f)}$ and 
its Euler characteristic $\chi_c(\overline{A_2(f_t)})$ 
is described by Milnor number of $\overline{A_2(f)}$ 
when $\overline{A_2(f)}$ has an isolated singularity at $0$.
This means $\#(A_3(f_t))$ is determined by $\mu(\overline{A_1(f)})$, 
$\mu(\overline{A_2(f)})$ and $\deg f$:
$$
\#(A_3(f_t))=\deg f-\mu(\overline{A_1(f)})+\mu(\overline{A_2(f)})-3,
$$
when $\overline{A_1(f)}$ and $\overline{A_2(f)}$ 
have isolated singularities at $0$.
%
\end{rem}
\begin{rem}
Consider a map germ $f:(\BB{C}^n,0)\to(\BB{C}^3,0)$, $n>3$. 
Take a stable perturbation $f_t$ of $f$. Then we obtain: 
$$
\chi_c(\overline{A_1(f_t)})+\chi_c (\overline{A_2(f_t)})+\#A_3(f_t)
=(-1)^n(1-\chi_f).
$$
Consider the map  
$F:(\BB{C}^n,0)\times(\BB{C},0)\to(\BB{C}^3,0)\times(\BB{C},0)$ 
defined by $F(x,t)=(f_t(x),t)$. 
Since $\overline{A_1(F)}$ is determinantal, it is Cohen-Macaulay.
We obtain that the map 
$\overline{A_1(F)}\to(\BB{C},0)$, $(x,t)\mapsto t$, is flat, 
and $\overline{A_1(f_t)}$ is a smoothing, which is determinantal.  
So the topology of $\overline{A_1(f_t)}$ is determined by
 $\overline{A_1(f)}$ when $\overline{A_1(f)}$ has isolated singularity
 at $0$.  
By Theorem 2.9 in \cite{Fukui-Weyman}, 
$\overline{A_2(F)}$ is Cohen-Macaulay if and only if $n=4,5$. 
We thus obtain that the map 
$\overline{A_2(F)}\to(\BB{C},0)$, $(x,t)\mapsto t$, is flat, 
if only if $n=4,5$.  
Assume that $n=4,5$. Then 
$\overline{A_2(f_t)}$ is a smoothing of $\overline{A_2(f)}$ and 
its Euler characteristic $\chi_c(\overline{A_2(f_t)})$ 
is described by the Milnor number of $\overline{A_2(f)}$:
$\chi_c(\overline{A_2(f_t)})=1-\mu(\overline{A_2(f)})$.
This means $\#(A_3(f_t))$ is determined by $\mu(\overline{A_1(f)})$, 
$\mu(\overline{A_2(f)})$ and $\chi_f$.
When $n\ge6$, we do not know whether
$\chi_c(\overline{A_2(f_t)})=1-\mu(\overline{A_2(f)})$ 
holds or not.  
%
%
\end{rem}
The following example also shows that the reduced structure 
of singularities locus may not fit the context of deformation of maps.  
\begin{exam}
Let us consider the image of the map $g:\BB{C}\to\BB{C}$ defined by 
$s\mapsto(s^3,s^4,s^5)$, which Milnor number $\mu$ is 4. 
The defining ideal is: 
$$
I_0=\langle xz-y^2,\ yz-x^3,\ x^2y-z^2\rangle.
$$
We know it defines a Cohen-Macaulay space. Consider the map: 
$$
G:(\BB{C}^2,0)\to(\BB{C}^4,0)\quad 
\textrm{defined by}\quad (s,t)\mapsto(x,y,z,t)
=(st+s^3,s^4,s^5,t). 
$$
Remark that $g_0(s)=g(s)$ where $G(s,t)=(g_t(s),t)$. 
The image of $g_t$, $t\ne0$, is nonsingular, 
and its Euler characteristic is 1, which is not $1-\mu$. 
Let us see what happens in this example. 
Eliminating $s$ from the ideal generated by: 
$$
x-st-s^3,\ y-s^4,\ z-s^5, 
$$ 
we obtain the ideal:
\begin{multline*}
I=\langle 
z^2-x^2y+ty^2+txz, 
xy^2-x^2z+tyz+t^2xy-t^3z,
y^3-xyz+tx^2y-t^2y^2+t^2xz,\\
xyz+tz^2-x^4+2tx^2y+2t^2xz+t^4y,
y^2z-xz^2+tx^2z-2t^2yz-t^3xy+t^4z
\rangle
\end{multline*}
of $\BB{C}\{x,y,z,t\}$. 
We remark that the variety $X$ defined by the ideal $I$ is not Cohen-Macaulay.
We also remark that this defines a reduced space, but the fiber
 $\pi^{-1}(0)$, 
where $\pi:X\to\BB{C}$ is the projection  $\pi(x,y,z,t)=t$,
is not reduced, since 
$$
\BB{C}\{x,y,z,t\}/I\otimes_{\BB{C}}\BB{C}\{t\}/\langle t\rangle 
\simeq
\BB{C}\{x,y,z\}/I_0\cap\langle x,\ y^3,\ y^2z,\ z^2\rangle. 
$$
\end{exam}

\end{document}